\documentclass[reqno,11pt,a4paper,twoside]{article}
\usepackage[paperheight=279.4mm,paperwidth=210mm,hmargin=1in,vmargin=1in]{geometry}
\usepackage{fourier}
\usepackage{ebgaramond}
\usepackage[font=small,labelfont=bf]{caption}
\usepackage[T1]{fontenc}
\usepackage{amsmath,amsthm}
\usepackage{graphicx}  
\usepackage{cases}
\numberwithin{equation}{section}
\newtheorem{theorem}{Theorem}[section]
\newtheorem{corollary}[theorem]{Corollary}
\newtheorem{lemma}[theorem]{Lemma}

\theoremstyle{definition}
\newtheorem{definition}[theorem]{Definition}
\newtheorem{example}[theorem]{Example}
\theoremstyle{remark}
\newtheorem{remark}[theorem]{Remark}

\let\originalleft\left
\let\originalright\right
\renewcommand{\left}{\mathopen{}\mathclose\bgroup\originalleft}
\renewcommand{\right}{\aftergroup\egroup\originalright}

\renewcommand{\tilde}{\widetilde}
\renewcommand{\hat}{\widehat}

\renewcommand{\Im}{\operatorname{Im}}
\newcommand{\dbold}{\boldsymbol{d}}
\newcommand{\Ecurly}{\mathcal{E}}
\newcommand{\fbold}{\boldsymbol{f}}

\newcommand{\Hbold}{\boldsymbol{H}}
\newcommand{\Lbold}{\boldsymbol{L}}
\newcommand{\Lcurlybold}{\boldsymbol{\mathcal{L}}}
\newcommand{\nbold}{\boldsymbol{n}}
\newcommand{\Pbold}{\boldsymbol{P}}
\newcommand{\Phibold}{\boldsymbol{\Phi}}

\newcommand{\Tbold}{\boldsymbol{T}}
\newcommand{\ubold}{\boldsymbol{u}}
\newcommand{\Ubold}{\boldsymbol{U}}
\newcommand{\vbold}{\boldsymbol{v}}

\newcommand{\xbold}{\boldsymbol{x}}
\newcommand{\ybold}{\boldsymbol{y}}
\newcommand{\zbold}{\boldsymbol{z}}
\newcommand{\zerobold}{\boldsymbol{0}}
\newcommand{\xibold}{\boldsymbol{\xi}}
\DeclareMathOperator{\Div}{div}
\DeclareMathOperator{\Grad}{\mathbf{grad}}
\DeclareMathOperator{\Curl}{\mathbf{curl}}
\DeclareMathOperator{\VectLapl}{\boldsymbol{\Delta}}
\newcommand{\Scal}[2]{\left\langle #1\right\rangle_{#2}}
\DeclareMathOperator{\Spec}{Spec}
\newcommand{\Neu}{\mathrm{Neu}}
\newcommand{\Dir}{\mathrm{Dir}}
\newcommand{\mix}{\mathrm{mix}}
\newcommand{\Jones}{\mathrm{Jones}}
\newcommand{\Rob}{\mathrm{Rob}}
\newcommand{\tang}{\mathrm{t}}
\newcommand{\normal}{\mathrm{n}}
\newcommand{\DtN}{\mathcal{D}}
\newcommand{\NtD}{\mathcal{N}}
\newcommand{\ip}{\Lambda}
\newcommand{\ipmix}{\Lambda_{\mix}}
\newcommand{\Ltworho}{\boldsymbol{L}_{\rho_s}^2(\Omega)}
\newcommand{\Hone}{\boldsymbol{H}^1(\Omega)}
\newcommand{\Pn}{P_\normal}
\renewcommand{\imath}{\mathrm{i}}
\newcommand{\email}[1]{\href{mailto: #1}{#1}}
\author{%
Michael Levitin%
\thanks{\textbf{ML}: Department of Mathematics and Statistics, University of Reading, Whiteknights, PO Box 220, 
Reading RG6 6AX, UK; \email{m.levitin@reading.ac.uk}; \href{http://www.michaellevitin.net}{www.michaellevitin.net}}
\and Peter Monk%
\thanks{\textbf{PM}: Department of Mathematical Sciences, University of Delaware,
Newark, DE 19716, USA; \email{monk@udel.edu}}
 \and Virginia Selgas%
 \thanks{\textbf{VS}: Departamento de Matem\'aticas, Universidad de Oviedo,
EPIG, C/ Luis Ortiz Berrocal s/n,
33203 Gij\'on, Spain; \email{selgasvirginia@uniovi.es}; \href{https://virginiaselgas.com}{virginiaselgas.com}}
}
\usepackage{fancyhdr}
\title{Impedance eigenvalues in linear elasticity
\footnote{{\bf Published in }SIAM J. Appl. Math. \textbf{81}(6) (2021), 2433--2456, doi  \href{https://doi.org/10.1137/21M1412955}{10.1137/21M1412955}.
This arXiv version corrects a misprint in formula \eqref{eq:potsol}  in the Supplementary materials in the published version}
\footnote{The research of P. Monk is partially supported by the US AFOSR under grant number FA9550-20-1-0024. The stay of V. Selgas in the University of Delaware was funded by the program Movilidades de Excelencia of the University  of Oviedo, and her research is partially supported by the project MTM2017-87162-P of MINECO. This research was initiated while M. L. and P. M. attended a workshop on ``Steklov eigenproblems'' at the American Institute of Mathematics}}
\renewcommand\footnotemark{}
\usepackage{appendix}
\usepackage{tocloft}
\advance\cftsecnumwidth 1em\relax
\advance\cftsubsecindent 1em\relax
\advance\cftsubsecnumwidth 1em\relax
\usepackage[sf,sl,outermarks]{titlesec}
\titleformat{\section}
{\normalfont\large\bfseries}
{\filcenter\thesection.}{1ex}{\filcenter}
\titleformat{\subsection}[runin]
{\normalfont\normalsize\bfseries}{\thesubsection.}{1ex}{}[.]
\usepackage[colorlinks,allcolors=blue]{hyperref}
\begin{document}
\date{28 January 2022}
\maketitle
\begin{abstract}
This paper is devoted to studying impedance eigenvalues (that is, eigenvalues of a particular Dirichlet-to-Neumann map)  for the time harmonic  linear elastic wave problem, and their potential use as target-signatures for fluid-solid interaction problems.  We first consider several possible families of eigenvalues of the elasticity problem, focusing on certain impedance eigenvalues that are an analogue of Steklov eigenvalues. We show that one of these families arises naturally in inverse scattering. We also analyse their approximation from far field measurements of the scattered pressure field in the fluid, and illustrate several alternative methods of approximation in the case of an isotropic elastic disk.  
\end{abstract}
%
\tableofcontents
\section{Introduction}\label{sec:introduction}

In classical scattering theory, \emph{target signatures} are discrete sets of numbers that can be computed from
scattering data and which can either help to characterise a scatterer (by comparing the computed signatures with a 
dictionary of signatures) or be used to indicate changes in a scatterer due to changes in the signature. The first electromagnetic target signatures were scattering resonances which can, in principle, be computed from a time domain radar signal~\cite{Melrose95}.  More recently, Cakoni, Colton and co-workers have suggested the use of transmission eigenvalues as target signatures for penetrable scatterers in the electromagnetic and acoustic contexts~\cite{col19}.  However, for an absorbing penetrable medium, real transmission eigenvalues do not exist~\cite[Theorem 8.12]{col19}, and hence they cannot be determined from scattering data by current methods.  This has led to a search for alternative target signatures with one approach being relevant to this paper: those known as \emph{Steklov eigenvalues}~\cite{cakoni+al2016,CLM73}. The usage of these eigenvalues as target signatures is based on the idea of modifying the far field operator.  This technique was first used in \cite{CM115,CM116} as part of a shape reconstruction algorithm precisely to avoid a breakdown of the method at transmission eigenvalues. 

In this paper we will consider the linearised fluid-solid interaction problem in which a solid is surrounded by fluid and interrogated by incident waves from the fluid.  Besides being interesting in its own right~\cite{hsi00,luke+martin,monk+selgas_1,monk+selgas_2} this problem involves multiphysics  and serves to illustrate some of
the issues that arise in the application of target signatures to more complex problems.  

Before studying the fluid-solid interaction problem, we first analyse several different eigenvalue problems for linear elasticity, focusing on eigenvalues for certain Steklov like problems. Because the equations of elasticity involve vector functions, there are several ways to define a Dirichlet-to-Neumann map, and hence several possible Steklov type eigenvalue problems. One of these  arises from the aforementioned consideration of target signatures. Because the new eigenvalue problem is non-standard, we distinguish it from the classical Steklov problem and refer to it as the \emph{impedance problem}.

Our paper makes novel contributions in two ways. First, we study several families of eigenvalues for the elasticity problem and provide new estimates regarding their parametric dependence.  Then, for one family, we show how to relate these eigenvalues to a modified far field equation and verify that they can be obtained from far field data.  Some numerical results, using a novel technique for finding eigenvalues,  illustrate our theory.  

The rest of this paper proceeds as follows. In Section \ref{sec:MichaelOnEigs} we study several eigenvalue problems in elasticity needed to describe our results, and of interest in their own right.  We prove parametric dependence and existence results. Then, in Section~\ref{sec:directFSproblem}, we summarise the forward fluid-solid interaction problem that underlies the inverse problem we shall consider. We also recall the definition and basic properties of the  far field operator, as well as define an auxiliary problem and its resulting modified far field operator.

Next, in Section \ref{sec:ffoperator_modification1} we study the inverse problem at hand: we show how impedance eigenvalues are related to solutions of the modified far field equation. We also discuss their numerical approximation by solving a parametrised set  of modified far field equations, so that they may be obtained from scattering data.

Finally, in Section \ref{sec:NumSteklovFS}, we investigate numerically the approximation of the impedance eigenvalues from far field measurements for a particular  two-dimensional case. This involves a new method for approximating
these eigenvalues based on a further modification of the modified far field equation.

Concerning notation, boldface quantities will represent vector valued functions or spaces.  In particular, $\Hone:=(H^1(\Omega))^m$ where $\Omega\subset\mathbb{R}^m$.

%
\section{Elasticity normal-normal Dirichlet-to-Neumann and Neumann-to-Diri\-ch\-let maps}\label{sec:MichaelOnEigs}

In this section we use appropriate Dirichlet-to-Neumann and Neu\-mann-to-Dirichlet operators to study 
interior eigenvalue problems, in particular of the impedance type, for the elasticity system.  First we define notation for the elasticity system under study.
 Later we will couple it to external equations, together with transmission conditions to obtain the fluid-solid interaction problem. 

\subsection{The elasticity system} 
We denote by $\Omega \subset \mathbb{R}^m$ (where $m=2$ or $3$) a bounded domain occupied by a linear elastic solid. Furthermore, $\Gamma  = \partial \Omega$ denotes the  boundary of $\Omega$, and $\nbold$ is the unit outward normal to $\Omega$ on $\Gamma$. For simplicity, we shall assume that $\Gamma$ is smooth.  Because we later want to consider the fluid-solid problem we also assume that $\mathbb{R}^m\setminus\overline{\Omega}$ is connected.
We suppose that the solid is isotropic, homogeneous and undertakes small deformations.  We also assume that the  elastodynamic displacement field, denoted $\ubold$, is time-harmonic, and then work in the frequency domain.

We denote by $\epsilon (\ubold)  := \left( \frac{1}{2} (\partial_i u_j + \partial_j u_i) \right)_{i,j = 1}^{m}$ the infinitesimal strain tensor and by $\sigma (\ubold) := \lambda \,\Div \ubold \; I + 2 \mu \, \epsilon (\ubold)$ the stress tensor; here and in the sequel, $I$ is the identity tensor (that is, $I := (\delta_{ij})_{i,j = 1}^{m}$ where $\delta_{ij}$ stands for the Kronecker delta) and $\lambda, \mu \in \mathbb{R}$ are the Lam\'e moduli. We also define the standard traction operator  by 
\begin{equation*}
\Tbold\ubold := \sigma (\ubold ) \, \nbold = \Big(\lambda \Div  \ubold\; n_i + 2 \mu \sum\limits_{j = 1}^{m} \epsilon_{ij} (\ubold )\;  n_j\Big)_{i = 1}^{m} \qquad\mbox{on }\Gamma \, .
\end{equation*} 
In the remainder of the paper, we make the following general assumptions (cf. \cite{luke+martin}) on these coefficients: the functions $\mu $ and $ \lambda + \frac{2}{m} \mu $ are  bounded, piecewise smooth and uniformly strictly positive in $\Omega$. 
In this section, to allow us to state results using the theory of pseudodifferential operators, we assume that $\lambda$ and $\mu$ are constant in $\Omega$.  In later sections they will be taken to be piecewise smooth. The key existence and discreteness results for impedance eigenvalues hold for more general coefficients $\lambda$ and $\mu$; it is also sufficient to assume that the boundary $\Gamma$ is Lipschitz, see \cite{AAL}.

The mass density  in the solid  is denoted by  $\rho_s: \Omega\to \mathbb{R}$  and is assumed to be a piecewise smooth real valued function such that $\rho_s(\xbold) \ge \rho_{s,0}>0$ a.e. in $\Omega$ (where $\rho_{s,0}$ is a constant). 

Under the previous hypotheses, the elastodynamic displacement field $\ubold$ satisfies the elasticity system in the time harmonic regime
\begin{equation}
\nabla \cdot \sigma (\ubold) + \rho_s \Lambda \ubold = \zerobold \mbox{ in }\Omega,\label{eq:PMa1}\\
\end{equation}
where $\Lambda=\omega^2$ and $\omega$ stands for the angular frequency of the wave.  To make the eigenvalue problems clearer, we allow $\ip$  to be any real number (possibly negative) so that, in this section, $\omega$ and hence $\ip$ is sometimes considered as a eigenvalue rather than as a fixed frequency. We will consider various boundary conditions that will be discussed as needed in the upcoming section.

\subsection{Boundary value problems for the elasticity operator}\label{subs:bvpelasticity} 
For convenience  we rewrite \eqref{eq:PMa1} as
\begin{equation}\label{eq:elasticvibrations}
\Lbold\ubold+ \rho_s \ip \ubold =\zerobold,
\end{equation}
where
\begin{equation}\label{eq:elasticoperator}
\Lbold\ubold=\Lbold_{\lambda,\mu}\ubold := \left(\mu\VectLapl+(\lambda+\mu)\Grad\Div \right)\ubold = 
\left(-\mu\Curl\Curl+(\lambda+2\mu)\Grad\Div\right)\ubold
\end{equation}
defines the standard linear elasticity operator.  

The following facts are standard, and are collected below mostly in order to fix notation. Multiplying $\Lbold\ubold$ by $\overline{\vbold}$ and integrating by parts, one obtains Green's formula, see, e.g., \cite{kupradze}:
\begin{equation}\label{eq:Lbyparts}
\Scal{\Lbold\ubold,\vbold}{\Lbold^2(\Omega)}=-\Ecurly[\ubold,\vbold]+\Scal{\Tbold\ubold,\vbold}{\Lbold^2(\Gamma)},
\end{equation}
where $\Scal{\ubold,\vbold}{\Lbold^2(\Omega)}$ and $\Scal{\ubold,\vbold}{\Lbold^2(\Gamma)}$ denote the standard inner products $\displaystyle\int_\Omega \ubold(\xbold)\cdot\overline{\vbold(\xbold)}\,d\xbold$ and $\displaystyle\int_\Gamma \ubold(\xbold)\cdot\overline{\vbold(\xbold)}\,dS_{\xbold}$, respectively, and
\[
\Ecurly[\ubold,\vbold]
:=\int_\Omega \left(\lambda\Div\ubold\Div\overline{\vbold}+\mu\sum_{i,j=1}^m \left(\frac{\partial u_{i}}{\partial x_{j}}+\frac{\partial u_{j}}{\partial x_{i}}\right)\overline{\frac{\partial v_{i}}{\partial x_{j}}}\right)\,d\xbold.
\]
For constant $\lambda$ and $\mu$, this can be rearranged as
\begin{equation}\label{eq:E}
\begin{split}
\Ecurly[\ubold,\vbold]
& =\lambda\Scal{\Div\ubold,\Div\vbold}{L^2(\Omega)}\\
&+\mu\left(-\Scal{\Curl\ubold,\Curl\vbold}{\Lbold^2(\Omega)}+2\sum_{j=1}^m \Scal{\Grad u_j,\Grad v_j}{\Lbold^2(\Omega)}\right) ,
\end{split}
\end{equation}
where we use the standard definition of  the  curl of a vector field in dimensions two or three. We immediately see from \eqref{eq:E} that $\Ecurly$ is Hermitian,
\[
\Ecurly[\ubold,\vbold]=\overline{\Ecurly[\vbold,\ubold]},
\]
for any $\ubold,\vbold\in \Hone$; moreover, for any $\ubold\in \Hone$ we have 
\[
\Ecurly[\ubold,\ubold]\ge 0,
\] 
and if, additionally, $\ubold$ solves \eqref{eq:elasticvibrations} in the weak sense, then
\[
\Ecurly[\ubold,\ubold]-\ip\Scal{\ubold,\ubold}{\Ltworho}=\Scal{\Tbold\ubold,\ubold}{\Lbold^2(\Gamma)};
\]
here $\Ltworho$ is the Hilbert space $\Lbold^2(\Omega)$ with the  weighted inner product  
\[
\Scal{\ubold,\vbold}{\Ltworho}:=\int_\Omega \rho_s(\xbold)\, \ubold(\xbold)\cdot\overline{\vbold(\xbold)}\,d\xbold.
\]

In what follows we  refer to five spectral boundary value problems for \eqref{eq:elasticvibrations}, treating $\ip$ as a spectral parameter. To be more precise, 
we deal with spectral boundary value problems for the operator pencil 
\[
\Lcurlybold=\Lcurlybold(\ip):=\Lbold+\rho_s\ip.
\]
The first is the standard \emph{Neumann eigenvalue problem} of finding $\ubold\in \boldsymbol{H}^1(\Omega)/\mathbb{R}$, $\ubold\not=\zerobold$, and $\ip\in\mathbb{R}$ such that
\begin{equation}\label{eq:elasticNeumann}
\begin{cases}
\Lcurlybold(\ip)\ubold=\zerobold\qquad&\text{in }\Omega,\\
\Tbold\ubold=\zerobold\qquad&\text{on }\Gamma.
\end{cases}
\end{equation}
Its spectrum $\Spec(\Lcurlybold^\Neu)=\{0=\ip_{\Neu,1}, \ip_{\Neu,2},\dots\}$  is discrete and consists of non-negative eigenvalues $\ip_{\Neu,j}$, repeated with multiplicities and enumerated  non-decreasingly,  with the only limit point at $+\infty$. The eigenvalues can be found using a standard minimax principle,
\begin{equation}\label{eq:varprN}
\ip_{\Neu,j}=\inf_{\substack{\mathcal{H}\subset \Hone\\\dim \mathcal{H}=j}}\ \sup_{\substack{\ubold\in\mathcal{H}\\\ubold\ne\zerobold}}\ \frac{\Ecurly[\ubold,\ubold]}{\|\ubold\|^2_{\Ltworho}} \qquad \text{for } j=1,2,\ldots
\end{equation}

The second  is the standard \emph{Dirichlet eigenvalue problem} (which is not needed in this paper but included for completeness) that seeks $\ubold\in \boldsymbol{H}^1(\Omega)$, $\ubold\not=\zerobold$, and $\ip\in\mathbb{R}$ such that
\[
\begin{cases}
\Lcurlybold(\ip)\ubold=\zerobold\qquad&\text{in }\Omega,\\
\ubold=\zerobold\qquad&\text{on }\Gamma.
\end{cases}
\]
Its spectrum $\Spec(\Lcurlybold^\Dir)=\{\ip_{\Dir,1}, \ip_{\Dir,2},\dots\}$  is discrete and consists of positive eigenvalues $\ip_{\Dir,j}$, repeated with multiplicities and enumerated  non-decreasingly,  with the only limit point at $+\infty$. The eigenvalues again can be found using a standard minimax principle,
\begin{equation}\label{eq:varprD}
\ip_{\Dir,j}=\inf_{\substack{\mathcal{H}\subset \boldsymbol{H}^1_0(\Omega)\\ \dim \mathcal{H}=j}}\ \sup_{\substack{\ubold\in\mathcal{H}\\\ubold\ne\zerobold}}\ \frac{\Ecurly[\ubold,\ubold]}{\|\ubold\|^2_{\Ltworho}} \qquad \text{for } j=1,2,\ldots 
\end{equation}

The third eigenvalue  problem  that we need has \emph{mixed boundary conditions}, with the Neumann conditions imposed tangentially to the boundary, and the Dirichlet conditions in the normal direction. More precisely, for a field $\fbold$ defined on $\Gamma$,  let $\Pn:=\fbold\mapsto(\fbold\cdot\nbold)$ and $\Pbold_\tang:=\fbold\mapsto \fbold-(\Pn\fbold)\nbold$ be the normal and tangential projection operators, respectively, so that 
\[
\fbold = \Pbold_\tang\fbold + \nbold \Pn\fbold.
\]
Consider the mixed spectral problem of finding $\ubold\in \boldsymbol{H}^1(\Omega) $, $\ubold\not=\zerobold$, and $\ip\in\mathbb{R}$ such that
\begin{equation}\label{eq:elasticmixed}
\begin{cases}
\Lcurlybold(\ip)\ubold=\zerobold\qquad&\text{in }\Omega,\\
\Pbold_\tang\Tbold\ubold=\zerobold\qquad&\text{on }\Gamma,\\
\Pn(\ubold|_{\Gamma})=0\qquad&\text{on }\Gamma.
\end{cases}
\end{equation}
Similar to the Neumann and Dirichlet problems, its spectrum $\Spec(\Lcurlybold^\mix)=\{\ip_{\mix,1}, \ip_{\mix,2},\dots\}$ consists of non-negative eigenvalues, repeated with multiplicities and enumerated non-decreasingly, and the minimax principle takes the form
\begin{equation}\label{eq:varprM}
\ip_{\mix,j}=\inf_{\substack{\mathcal{H}\subset \Hbold^1_{\normal,0}(\Omega)\\\dim \mathcal{H}=j}}\ \sup_{\substack{\ubold\in\mathcal{H}\\\ubold\ne\zerobold}}\ \frac{\Ecurly[\ubold,\ubold]}{\|\ubold\|^2_{\Ltworho}} \qquad \text{for } j=1,2,\ldots,
\end{equation}
where the Sobolev space $\Hbold^1_{\normal,0}(\Omega)\subset \Hone$ consists of vector-valued functions from $\Hone$ whose normal components vanish on the boundary.

The fourth eigenvalue problem is the \emph{mixed Neumann--Robin problem},
\begin{equation}\label{eq:elasticRobin}
\begin{cases}
\Lcurlybold(\ip)\ubold=\zerobold\qquad&\text{in }\Omega,\\
\Pbold_\tang\Tbold\ubold=\zerobold\qquad&\text{on }\Gamma,\\
\Pn(\Tbold\ubold)=\kappa\Pn(\ubold|_{\Gamma})\qquad&\text{on }\Gamma,
\end{cases}
\end{equation}
where $\kappa\in\mathbb{R}$ is a given parameter.  
Its spectrum $\Spec(\Lcurlybold^{\Rob(\kappa)})=\{\ip_{\Rob(\kappa),1}, \ip_{\Rob(\kappa),2}, \dots\}$, written once more as a multiset of eigenvalues in non-decreasing order with multiplicities, is given by the minimax principle
\begin{equation}\label{eq:varprR}
\ip_{\Rob(\kappa),j}=\inf_{\substack{\mathcal{H}\subset \Hbold^1(\Omega)\\\dim \mathcal{H}=j}}\ \sup_{\substack{\ubold\in\mathcal{H}\\\ubold\ne\zerobold}}\ 
\frac{\Ecurly[\ubold,\ubold]-\kappa\|P_n\ubold\|^2_{L^2(\Gamma)}}{\|\ubold\|^2_{\Ltworho}} \qquad \text{for } j=1,2,\ldots.
\end{equation}
We note that 
\[
\Lcurlybold^{\Rob(0)}=\Lcurlybold^\Neu,
\]
and (at least formally)
\[
\Lcurlybold^{\Rob(-\infty)}=\Lcurlybold^\mix,
\]
and that the eigenvalues $\ip_{\Rob(\kappa),j}$ are monotone non-increasing in $\kappa$.

Using the fact that 
\[
\boldsymbol{H}^1_0(\Omega)\subset \Hbold^1_{\normal,0}(\Omega)\subset \Hone,
\]
and the variational principles \eqref{eq:varprN}, \eqref{eq:varprD}, \eqref{eq:varprM}, and \eqref{eq:varprR},
we immediately obtain, with $\kappa_1<0<\kappa_2$, the bounds 
\[
\ip_{\Rob(\kappa_2),j}\le\ip_{\Neu,j}=\ip_{\Rob(0),j}\le\ip_{\Rob(\kappa_1),j}\le\ip_{\mix,j}=\ip_{\Rob(-\infty),j}\le\ip_{\Dir,j}
\]
for $j =1,2,\ldots$. 

We will denote, for $\aleph\in\{\Neu,\Dir,\mix\}$, the standard counting functions of eigenvalues of $\Lcurlybold^\aleph$ less than a given $\Lambda$, by
\begin{equation}\label{eq:NLambda}
N^\aleph(\Lambda):=\#\{j; \ip_{\aleph,j}<\Lambda\}.
\end{equation}

Finally, consider the following  {overdetermined} eigenvalue problem for the \emph{Jones modes}.  In this problem we seek $\ubold\in \boldsymbol{H}^1(\Omega)$, $\ubold\not=0$, and $\ip\in\mathbb{R}$ such that
\begin{equation}\label{eq:elasticJones}
\begin{cases}
\Lcurlybold(\ip)\ubold=\zerobold\qquad&\text{in }\Omega,\\
\Tbold\ubold=\zerobold\qquad&\text{on }\Gamma,\\
\Pn(\ubold|_{\Gamma})=0\qquad&\text{on }\Gamma,
\end{cases}
\end{equation}
and denote its spectrum  by $\Spec(\Lcurlybold^\Jones)$. For any $\ip\in \Spec(\Lcurlybold^\Jones)$ with $\ip>0$, we call $\omega=\ip^{1/2}$ a \emph{Jones frequency} for $\Omega$, and the corresponding non-trivial solutions $\ubold\in\Hbold^1_{\normal,0}(\Omega)$ of \eqref{eq:elasticJones} are called \emph{Jones modes}. It is a classical result that Jones frequencies exist for axisymmetric bodies. Moreover, when they exist
the set of the Jones frequencies is discrete, and each Jones frequency has a finite multiplicity (see \cite[\S 3]{luke+martin}, \cite[\S 2.1]{natroshvili+kharibegashvili+tediashvili} and references therein.

A comparison of \eqref{eq:elasticNeumann}, \eqref{eq:elasticmixed}, and \eqref{eq:elasticJones} shows that
\[
\Spec(\Lcurlybold^\Jones)\subseteq \Spec(\Lcurlybold^\Neu)\cap \Spec(\Lcurlybold^\mix);
\]
moreover, we note that $\omega$ is a Jones frequency if and only if \eqref{eq:elasticNeumann} and \eqref{eq:elasticmixed} have a common non-trivial eigenmode $\ubold$ corresponding to the same eigenvalue $\ip$.

\subsection{Normal-normal Dirichlet-to-Neumann map for elasticity}\label{subs:DtNmap} 
Let us now fix  $\Lambda\in\mathbb{R}$, for the moment such that $\ip\not\in\Spec({\Lcurlybold^\mix})$, and consider the non-homogenous boundary value problem associated to \eqref{eq:elasticmixed}:
\begin{equation}\label{eq:elasticmixednonhom}
\begin{cases}
\Lcurlybold(\ip)\ubold=\zerobold\qquad&\text{in }\Omega,\\
\Pbold_\tang\Tbold\ubold=\zerobold\qquad&\text{on }\Gamma,\\
\Pn(\ubold|_{\Gamma}) =f\qquad&\text{on }\Gamma.
\end{cases}
\end{equation}
For a given function $f\in H^{1/2}(\Gamma)$ this problem has a unique solution $\ubold =\ubold_f \in \Hone${; notice that $\Tbold\ubold_f =(\Tbold\ubold_f\cdot \nbold)\nbold\in H^{-1/2}(\Gamma)$}. We can therefore define the \emph{normal-normal Dirichlet-to-Neumann map}  (abbreviated as \emph{nnDtN map}), $\DtN_{\ip}$ which sends the normal Dirichlet datum $f=\Pn (\ubold|_{\Gamma})$  of a solution $\ubold$ of \eqref{eq:elasticmixednonhom} into its normal Neumann datum $\Tbold\ubold\cdot \nbold$ {(or, in physical terms, the normal component of the displacements into a normal traction)}.
By \eqref{eq:Lbyparts} and \eqref{eq:elasticmixednonhom}, we have
\begin{equation}\label{eq:DtNbyparts}
\Scal{\DtN_{\ip}f,g}{L^2(\Gamma)}=\Ecurly[\ubold_f,\ubold_g]-\ip\Scal{\ubold_f,\ubold_g}{\Lbold_{\rho_s}^2(\Omega)}
\end{equation}
for $f,g\in H^{1/2}(\Gamma)$, therefore $\DtN_{\ip}$ can be extended by duality to a self-adjoint operator acting in $L^2(\Gamma)$. 

If $\ip\in \Spec({\Lcurlybold^\mix})$, we can still define $\DtN_{\ip}$ in a similar manner, as long as it is restricted to the subspace of  $H^{1/2}(\Gamma)$ which is an orthogonal complement in $L^2(\Gamma)$ to $\Pn\Tbold\Phibold_{\ip}(\Lcurlybold^\mix)$,  where $\Phibold_{\ip}(\Lcurlybold^\mix)$ denotes the eigenspace of $\Lcurlybold^\mix$ corresponding to the eigenvalue $\ip$.

An important property of the nnDtN map for elasticity, similar to that of the standard Dirichlet-to-Neumann map for the scalar Laplacian, is given in the following lemma.
\begin{lemma}\label{lem:DtNprinsymb}
The nnDtN map $\DtN_{\ip}$ is an elliptic pseudodifferential operator of order one with the principal symbol 
\[
(\operatorname{prin\,symb} \DtN_{\ip})(\xibold)=\frac{2\mu(\lambda+\mu)}{\lambda+2\mu}|\xibold|,\qquad\xibold\in\mathbb{R}^{m-1}\setminus\{\zerobold\}.
\]
\end{lemma}
\begin{remark} We emphasise that the statement of Lemma \ref{lem:DtNprinsymb} only holds for a smooth boundary $\Gamma$. In that case, since the principal symbol of the scalar Laplace--Beltrami operator $-\Delta_\Gamma$ acting on $\Gamma$ is $|\xi|^2$, we can see that (modulo lower order  terms) the nnDtN map for elasticity behaves as $ \sqrt{-\Delta_\Gamma}$ up to a multiplicative constant.
\end{remark}
\begin{proof} This is fairly standard, and can be obtained from the principal symbol of the full Dirichlet-to-Neumann map for elasticity, $\ubold|_\Gamma\to \Tbold\ubold$ subject to $\Lcurlybold(\ip)\ubold=\zerobold$,  see, e.g., \cite{AAL}. We deduce the result directly for completeness in the following manner. Consider problem \eqref{eq:elasticmixednonhom} in the half space $\Omega=\mathbb{R}^{m-1}\times(-\infty,0)$ with the boundary $\Gamma=\mathbb{R}^{m-1}$ and the exterior unit normal $\nbold=(0,0,1)^T$. Now replace it with an ODE matrix boundary value problem
\begin{equation}\label{eq:elasticmixednonhomODE}
\begin{cases}
\Lbold\left(\imath\xibold,\frac{d}{dx_m}\right)\ubold=\zerobold,\qquad&\text{in }\Omega,\\
\Pbold_\tang\Tbold\left(\imath\xibold,\frac{d}{dx_m}\right)\ubold=\zerobold\qquad&\text{on }\Gamma,\\
\Pn(\ubold|_{\Gamma})=f\qquad&\text{on }\Gamma,
\end{cases}
\end{equation}
where we first of all have dropped the lower order term in $\Lcurlybold$, and then replaced the partial differential operators $\Lbold$ and $\Tbold$ by their ODE analogues  $\Lbold\left(\imath\xibold,\frac{d}{dx_m}\right)$ and $\Tbold\left(\imath\xibold,\frac{d}{dx_m}\right)$, respectively, with $\xibold\in\mathbb{R}^{m-1}$, and in which each differentiation with respect to $x_j$ ($j=1, \dots, m-1$) is replaced by multiplication by $\imath\xi_j$; notice that differentiation with respect to $x_m$ is preserved.  We then solve \eqref{eq:elasticmixednonhomODE} in the half-space looking for a solution $\ubold(\xibold, x_m)$ such that $\lim\limits_{x_m\to-\infty} \ubold=\zerobold$; then the principal symbol of $\ \DtN_{\ip}$ (which is in fact independent of $\omega$) is obtained from the relation on $x_m=0$, see e.g. \cite{SaVa, Tay},
 \[
\Pn\Tbold\left(\imath\xibold,\frac{d}{dx_m}\right)\ubold=(\operatorname{prin\,symb} \DtN_{\ip})(\xibold) f.
\]
For example, in the three-dimensional case the solution of \eqref{eq:elasticmixednonhomODE} is given by 
\[
\ubold(\xibold, x_3) =
\frac{f\mathrm{e}^{x_3 |\xibold|}}{2\mu(\lambda+\mu)}
\begin{pmatrix}
-\imath (\lambda+\mu)\xi_1 x_3\\
-\imath (\lambda+\mu)\xi_2 x_3\\
- (\lambda+\mu) x_3+\frac{\lambda+3\mu}{ |\xibold|}
\end{pmatrix},
\]
and the result follows by applying $\Tbold$ on $x_3=0$.
\end{proof}

As a corollary of Lemma \ref{lem:DtNprinsymb} and the self-adjointness of $\DtN_{\ip}$, we immediately obtain that its spectrum $\Spec(\DtN_{\ip})$ is discrete, semi-bounded below, and consists of isolated eigenvalues, $\Spec(\DtN_{\ip})=\{\alpha_{\DtN_{\ip}, 1}, \alpha_{\DtN_{\ip}, 2},\dots\}$, counted with multiplicities,  with the only accumulating point at $+\infty$. For $\ip\not\in\Spec({\Lcurlybold^\mix})$, these eigenvalues can be found using the minimax principle
\begin{equation}\label{eq:DtNminmax}
\alpha_{\DtN_{\ip}, j}=\inf_{\substack{\mathcal{H}\subset H^{1/2}(\Gamma)\\\dim \mathcal{H}=j}}\ \sup_{\substack{f\in\mathcal{H}\\f\ne 0}}\ \frac{\Scal{\DtN_{\ip}f,f}{L^2(\Gamma)}}{\|f\|^2_{L^2(\Gamma)}}.
\end{equation}
For $\ip\in\Spec(\Lcurlybold^\mix)$, one should restrict the spaces of admissible test-functions by requesting additionally $f=\Pn\ubold|_{\Gamma}$ to be orthogonal to $\Pn\Tbold\Phibold_{\ip}(\Lcurlybold^\mix)$ as discussed above.

Using  \eqref{eq:DtNbyparts}, we can re-state \eqref{eq:DtNminmax} as
\begin{equation}\label{eq:DtNminmax1}
\alpha_{\DtN_{\ip}, j}=\inf_{\substack{\tilde{\mathcal{H}}\subset \Hbold^{1}(\Lambda, \Gamma)\\\dim \tilde{\mathcal{H}}=j}}\ \sup_{\substack{\ubold\in\tilde{\mathcal{H}}\\\ubold\cdot\nbold\ne 0}}\ \frac{\Ecurly[\ubold,\ubold]-\ip\|\ubold\|^2_{\Lbold_{\rho_s}^2(\Omega)}}{\|\ubold\cdot \nbold\|^2_{L^2(\Gamma)}},
\end{equation}
where 
\[
\Hbold^1(\Lambda, \Gamma):=\{\ubold\in \Hbold^1(\Gamma): \Lcurlybold(\ip)\ubold=0\}.
\]

\begin{remark} For $\ip<\ip_{\mix,1}$ we can in fact further simplify \eqref{eq:DtNminmax1} by  replacing $\Hbold^{1}(\Lambda, \Gamma)$ in its statement with $\Hbold^1(\Gamma)$. This follows from the following simple observation: for any $\ip\in\mathbb{R}\setminus\Spec(\Lcurlybold^\mix)$ the space $\Hbold^1(\Gamma)$ can be decomposed into the direct (but not orthogonal) sum 
\[
\Hbold^1(\Gamma)= \Hbold^1(\Lambda, \Gamma)+ \Hbold^1_{\normal,0}(\Omega).
\]
Let us replace $\ubold\in  \Hbold^{1}(\Lambda, \Gamma)$ in  \eqref{eq:DtNminmax1} by $\ubold+\vbold\in \Hbold^1(\Gamma)$, where $\vbold\in \Hbold^1_{\normal,0}(\Omega)$. The denominator does not change, and the numerator after an integration by parts becomes
\[
\Ecurly[\ubold+\vbold,\ubold+\vbold]-\ip\|\ubold+\vbold\|^2_{\Lbold_{\rho_s}^2(\Omega)}=\left(\Ecurly[\ubold,\ubold]-\ip\|\ubold\|^2_{\Lbold_{\rho_s}^2(\Omega)}\right)+\left(\Ecurly[\vbold,\vbold]-\ip\|\vbold\|^2_{\Lbold_{\rho_s}^2(\Omega)}\right).
\]
The second term in the right-hand side is greater than or equal to $\left(\ip_{\mix,1}-\ip\right)\|\vbold\|^2_{\Lbold_{\rho_s}^2(\Omega)}$ by \eqref{eq:varprM}, which is in turn non-negative for $\vbold\in\Hbold^1_{\normal,0}(\Omega)\setminus\{\zerobold\}$ by our assumption on $\ip$. The minimisation then forces $\vbold=\zerobold$.
\end{remark}

We are now interested in the dependence of eigenvalues $\alpha_{\DtN_{\ip}, j}$ of $\DtN_{\ip}$ on the parameter $\ip$. The following result is almost a direct analogue for the corresponding result of Friedlander \cite{Friedlander} in the scalar case, see also a  further discussion in \cite{ArMa12} which in particular relaxes some of the smoothness conditions in  \cite{Friedlander}. The result can be also deduced from an abstract scheme of Safarov \cite{Saf08}. For an analogue for the full (matrix) DtN map in elasticity see \cite{Ann}.

\begin{lemma}\label{lem:FriedDtN} Assume that the Jones spectrum $\Spec(\Lcurlybold^\Jones)$ is empty. We have:
\begin{enumerate}
\item[{\normalfont (a)}] In every open interval of the $\ip$-real line not containing the points of  $\Spec({\Lcurlybold^\mix})$, each eigenvalue $\alpha_{\DtN_{\ip}, j}$ of $\DtN_{\ip}$ is a monotone decreasing continuous function of $\ip$.
\item[{\normalfont (b)}] Let $\ip_\Neu$ be an eigenvalue of multiplicity $M\ge 1$ of the Neumann elasticity problem \eqref{eq:elasticNeumann}.  
Then, exactly $M$ eigenvalue curves $\alpha_{\DtN_{\ip}}$ cross the line $\alpha=0$ from the positive into the negative half-plane at $\ip=\ip_{Neu}$.
\item[{\normalfont (c)}] Let $\ipmix$ denote an eigenvalue of multiplicity $M\ge 1$ of the mixed elasticity problem \eqref{eq:elasticmixed}. Then, exactly $M$ eigenvalue curves $\alpha_{\DtN_{\ip}}$ ``blow down'' to $-\infty$ as $\ip$ approaches $\ipmix$ from the left and ``blow up'' to $+\infty$ as $\ip$ approaches  $\ipmix$ from the right.
\end{enumerate}
\end{lemma}

We will give the proof of parts (a) and (b) below, and postpone the proof of part (c) until the next subsection; in proving (b) and (c) we assume for simplicity the stronger condition $\Spec(\Lcurlybold^\Neu)\cap \Spec(\Lcurlybold^\mix)=\emptyset$. 

\begin{remark}\label{rem:Jones} A more precise version of statements (b) and (c) without an assumption of the absence of Jones eigenvalues is easily adapted from Friedlander's arguments and reads as follows: Let, for $\ip\in\mathbb{R}$, denote by $M_{\ip}^\Neu$, $M_{\ip}^\mix$ and  $M_{\ip}^\Jones$ the multiplicities of $\ip$ as an eigenvalue of $\Lcurlybold^\Neu$, $\Lcurlybold^\mix$ and $\Lcurlybold^{\Jones}$, respectively, where either multiplicity may be zero if $\ip$ is not a corresponding eigenvalue. Additionally, let
\[
N_-(\ip):=\#\left(\Spec(\DtN_{\ip})\cap(-\infty,0)\right)
\]
denote the number of \emph{negative} eigenvalues of $\DtN_{\ip}$. Then an analogue of \cite[Lemma 2.4]{Friedlander} states that one-sided limits
\[
N_-({\Lambda_0}\pm 0):=\lim\limits_{\ip\to{\Lambda_0}\pm0} N_-(\ip)
\]
exist for all ${\Lambda}_0\in\mathbb{R}$, and 
\[
N_-({\Lambda_0}+0)=N_-({\Lambda_0}-0)-M_{{\Lambda_0}}^\mix+M_{{\Lambda_0}}^\Neu.
\]
As an additional corollary, we immediately obtain the relation between the counting function  $N_-$ and the counting functions $N^\Neu$ and $N^\mix$ introduced in \eqref{eq:NLambda}: for any $\ip\in\mathbb{R}$, 
\[
N_-(\ip)=N^\Neu(\ip)-N^\mix(\ip).
\]
Moreover, at each eigenvalue $\ip\in\Spec(\Lcurlybold^\Neu)\cup\Spec(\Lcurlybold^\mix)$ of either the Neumann elasticity problem or the mixed elasticity problem, exactly $M_{\ip}^\Neu-M_{\ip}^\Jones$ eigenvalue curves $\alpha_{\DtN_{\ip}}$ cross the line $\alpha=0$ from the positive into the negative half-plane, and exactly $M_{\ip}^\mix-M_{\ip}^\Jones$ curves ``blow down'' to $-\infty$ as $\ip$ approaches  $\Lambda_0$ from the left and ``blow up'' to $+\infty$ as $\ip$ approaches  ${\Lambda_0}$ from the right. 

For an illustration of the behaviour of the eigenvalues of the nnDtN map for a disk see Figures \ref{fig:fig1a} and \ref{fig:fig1b}, with explicit expression for the eigenvalues given in \S\ref{app:A}. 
\end{remark}

\begin{example} 
We further illustrate Remark \ref{rem:Jones} by looking at the value $\Lambda=0$ in the case of a unit disk. We see that $\ip=0$ is the eigenvalue of multiplicity three of the Neumann problem $\Lcurlybold^\Neu$, and the eigenvalue of multiplicity one of the mixed problem $\Lcurlybold^\mix$. The eigenmodes of $\Lcurlybold^\Neu$ corresponding to the eigenvalue $\Lambda=0$ are
\[
\mathbf{e}_\theta\qquad\text{and}\qquad (\mathbf{e}_r\pm\mathrm{i} \mathbf{e}_\theta)\mathrm{e}^{\pm \mathrm{i}\theta},
\]
in polar coordinates $(r,\theta)$ with the coordinate vectors $\mathbf{e}_r$ and $\mathbf{e}_\theta$. The first of these eigenvectors is simultaneously an eigenvector of the mixed problem $\Lcurlybold^\mix$, and the other two are not. Therefore $\Lambda=0$ is a Jones eigenvalue of multiplicity one for the disk; according to Remark \ref{rem:Jones}, we conclude that exactly $M_{0}^\Neu-M_{0}^\Jones=2$ eigenvalue curves $\alpha_{\DtN_{\ip}}$ cross the line $\alpha=0$ from the positive into the negative half-plane at $\Lambda=0$, and  (as $M_{\ip}^\mix-M_{\ip}^\Jones=0$) no eigenvalue curves blow up there, cf. Figure  \ref{fig:fig1b}. 
\end{example}

\begin{figure}[!hbt]
\begin{center}
\includegraphics{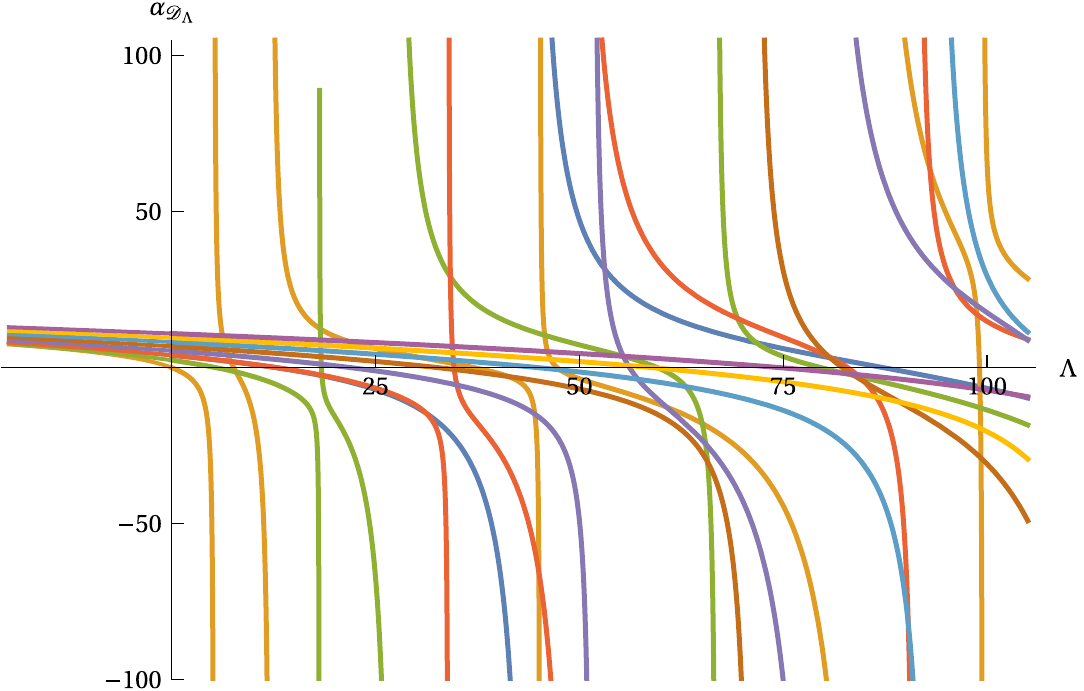}
\caption{Some eigenvalues $\alpha_{\DtN_{\ip}}$ of the nnDtN map for the unit disk as functions of $\ip$. Here $\lambda=\mu=\rho_s=1$.\label{fig:fig1a}} 
\end{center}
\end{figure}

\begin{figure}[!hbt]
\begin{center}
\includegraphics[width=0.95\textwidth]{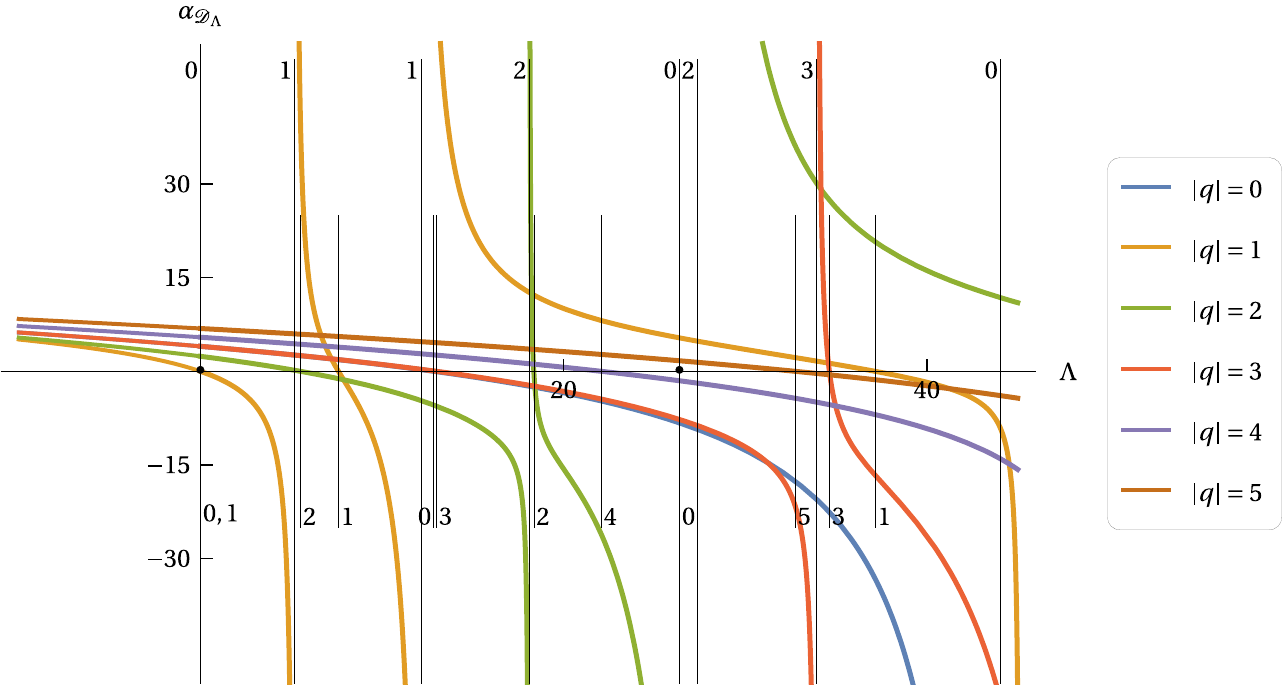}
\caption{Zoom of Figure \ref{fig:fig1a} that shows eigenvalues $\alpha_{\DtN_{\ip}}$ with the angular momentum $|q|\le 5$, see \S\ref{app:A} for details. The short vertical lines indicate the positions of the Neumann eigenvalues $\ip_\Neu$ (labelled with a corresponding angular momentum $|q|$), where the  curves  $\alpha_{\DtN_{\ip}}$ corresponding to the eigenmodes with the same angular momentum cross the line $\ip=0$. The long vertical lines indicate the positions of the mixed eigenvalues $\ip_\mix$ (labelled with a corresponding angular momentum $|q|$), where the  curves  $\alpha_{\DtN_{\ip}}$ corresponding to the eigenmodes with the same angular momentum blow up. All the curves and eigenvalues corresponding to $|q|>0$ are in fact double. We note that $\ip=0$ and $\ip\approx26.3746$, marked with black dots, are Jones eigenvalues for the disk with the angular momentum of the corresponding eigenmode being $q=0$; therefore  the corresponding curves $\alpha_{\DtN_{\ip}}$ have neither a zero nor a singularity at these values of $\ip$, see Remark \ref{rem:Jones}. \label{fig:fig1b}} 
\end{center}
\end{figure}

\begin{proof}[Proof of Lemma \ref{lem:FriedDtN}]
We prove the statement in part (a) by mimicking the reasoning in \cite{ArMa12}; a different approach similar to \cite{Friedlander} works as well.  We observe the following duality between the spectral problem 
\begin{equation}\label{eq:specforDtN}
\begin{cases}
\Lcurlybold(\ip)\ubold=\zerobold\qquad&\text{in }\Omega,\\
\Pbold_\tang\Tbold\ubold=\zerobold\qquad&\text{on }\Gamma,\\
\Pn\Tbold\ubold=\alpha \Pn(\ubold|_{\Gamma})\qquad&\text{on }\Gamma.
\end{cases}
\end{equation}
for the nnDtN map $\DtN_{\ip}$ (where $\ip$ is fixed and $\alpha$ is treated  as the spectral parameter) and the mixed Robin spectral problem \eqref{eq:elasticRobin} (where $\kappa$ is fixed and $\ip$ is treated as the spectral parameter). Namely, we have, for $\Lambda\not\in\Spec(\Lcurlybold^\mix)$,
\[
\alpha\in\Spec(\DtN_\ip)\qquad\iff\qquad \ip\in  \Spec(\Lcurlybold^{\Rob(\alpha)}). 
\] 
It is also easy to check the dimensions of the corresponding eigenspaces coincide.  Since the mixed Robin eigenvalues $\ip_{\Rob(\alpha), j}$ are non-increasing in $\alpha$, it immediately follows that the nnDtN eigenvalues $\alpha_{\DtN_\ip, j}$ are non-increasing in $\ip$ in each interval not containing points of $\Spec(\Lcurlybold^\mix)$. To prove the strict monotonicity, assume for contradiction that for some $j$ and $\ip_1^*<\ip_2^*$ we have 
\[
\alpha_{\DtN_\ip, j}=\alpha^*=\operatorname{const}\qquad\text{for }\ip\in[\ip_1^*,\ip_2^*].
\]
But then by duality 
\[
[\ip_1^*,\ip_2^*]\subset \Spec(\Lcurlybold^{\Rob(\alpha^*)}),
\]
which is impossible since the spectrum of $\Lcurlybold^{\Rob(\alpha^*)}$ is discrete. 

To prove (b), we compare  \eqref{eq:specforDtN} with \eqref{eq:elasticNeumann}: $\alpha=0$ is an eigenvalue of  $\DtN_{\ip}$ if and only if $\ip$ is an eigenvalue of $\Lcurlybold^\Neu$, and $\ubold$ is a corresponding eigenfunction.

Part (c) will follow from Lemma \ref{lem:FriedNtD} detailed below.
\end{proof}

\subsection{Normal-normal Neumann-to-Dirichlet map for elasticity}\label{subs:NtDmap}
By analogy with \S\ref{subs:DtNmap}, let us fix  $\ip\in\mathbb{R}$, for the moment such that $\ip\not\in\Spec(\Lcurlybold^\Neu)$ and consider  the following non-homogenous boundary value problem associated to \eqref{eq:elasticNeumann}:
\begin{equation}\label{eq:elasticNeunonhom}
\begin{cases}
\Lcurlybold(\ip)\ubold=\zerobold\qquad&\text{in }\Omega,\\
\Pbold_\tang\Tbold\ubold=\zerobold\qquad&\text{on }\Gamma,\\
\Pn\Tbold\ubold=f\qquad&\text{on }\Gamma.
\end{cases}
\end{equation}
For a given function $f\in H^{-1/2}(\Gamma)$ this problem has a unique solution $\ubold=\ubold_f\in \Hone$.  We can therefore define the \emph{normal-normal Neumann-to-Dirichlet map}  (abbreviated as \emph{nnNtD map}), $\NtD_{\ip}$ which sends the normal Neumann datum $f=\Tbold\ubold\cdot \nbold|_\Gamma$  of a solution $\ubold$ of \eqref{eq:elasticNeunonhom} into its normal Dirichlet datum $\Pn\ubold=\ubold\cdot \nbold|_\Gamma$. 
If $\ip\in \Spec({\Lcurlybold^\Neu})$, we can still define $\NtD_{\ip}$ in a similar manner, as long as  
it is restricted to the subspace of  $H^{-1/2}(\Gamma)$ which is an orthogonal complement in $L^2(\Gamma)$ to $\Pn\Phibold_{\ip}(\Lcurlybold^\Neu)|_\Gamma$.

For $\ip\not\in\Spec(\Lcurlybold^\Neu)\cup \Spec(\Lcurlybold^\mix)$, the nnDtN and the nnNtD maps are inverses of each other,
\[
\NtD_{\ip}\circ \DtN_{\ip}=\DtN_{\ip}\circ \NtD_{\ip}=\operatorname{Id}.
\]
This also holds for all $\ip$  if restricted to the corresponding domains whenever necessary. 
This fact, together with Lemmas \ref{lem:DtNprinsymb} and \ref{lem:FriedDtN},  imply the following result:
 \begin{lemma}\label{lem:FriedNtD} We have:
\begin{enumerate}
\item[{\normalfont (a)}] The nnNtD map $\NtD_{\ip}$ is an elliptic pseudodifferential operator of order minus one with the principal symbol 
\[
(\operatorname{prin\,symb} \NtD_{\ip})(\xibold)=\frac{\lambda+2\mu}{2\mu(\lambda+\mu)}\frac{1}{|\xibold|},\qquad\xibold\in\mathbb{R}^{m-1}\setminus\{\zerobold\}.
\]
\item[{\normalfont (b)}] For each $\ip\in\mathbb{R}$, the spectrum $\Spec(\NtD_{\ip})$ consists of isolated eigenvalues of finite multiplicity with the only accumulation point at $+0$.
\item[{\normalfont (c)}] In every open interval of the $\ip$-real line not containing the points of  $\Spec({\Lcurlybold^\Neu})$, each eigenvalue $\alpha_{\NtD_{\ip}, j}$ of $\NtD_{\ip}$ is a monotone increasing continuous function of $\ip$.
\item[{\normalfont (d)}] Let $\ipmix$ stand for an eigenvalue of multiplicity $M\ge 1$ of the mixed elasticity problem  \eqref{eq:elasticmixed}. 
Then, exactly $M$ eigenvalue curves $\alpha_{\NtD_{\ip}}$ cross the line $\alpha=0$ from the negative into the positive half-plane at $\ip=\ipmix$.
\item[{\normalfont (e)}] Let  $\ip_\Neu$ be an eigenvalue  of multiplicity $M\ge 1$ of the Neumann elasticity problem  \eqref{eq:elasticNeumann}. Then exactly $M$ eigenvalue curves $\alpha_{\NtD_{\ip}}$ ``blow up'' to $+\infty$ as $\ip$ approaches $\ip _\Neu$ from the left  
 and ``blow down'' to $-\infty$ as $\ip$ approaches $\ip _\Neu$ from the right. 
\end{enumerate}
 \end{lemma}
 
 \begin{remark} In what follows, we will be mostly interested in the eigenvalues of the nnNtD map; for brevity, we will from now on call them \emph{elasticity impedance eigenvalues}. We note that in some previous works in scalar context, the eigenvalues of the Neumann-to-Dirichlet map were called \emph{Steklov eigenvalues}. Traditionally, this is not entirely correct as this term is reserved for the eigenvalues  of the Dirichlet-to-Neumann map, strictly speaking also with $\Lambda=0$, and we will not use this terminology to avoid confusion.
\end{remark}

 \begin{proof}[Proof of Lemma \ref{lem:FriedNtD}] Since $\DtN_{\ip}$ and $\NtD_{\ip}$ are inverses of one another, parts (a)--(b) follow immediately from Lemma \ref{lem:DtNprinsymb}, part (c) from Lemma \ref{lem:FriedDtN}(a), and part (e) from Lemma  \ref{lem:FriedDtN}(b). To prove part (d), we write down the spectral problem for $\NtD_{\ip}$ explicitly, using $\alpha$ as a spectral parameter:
\[
\begin{cases}
\Lcurlybold(\ip)\ubold=\zerobold\qquad&\text{in }\Omega,\\
\Pbold_\tang\Tbold\ubold=\zerobold\qquad&\text{on }\Gamma,\\
\Pn\ubold=\alpha \Pn\Tbold\ubold\qquad&\text{on }\Gamma.
\end{cases}
\]
Then  $\alpha=0$ if and only if $\ip\in\Spec(\Lcurlybold^\mix)$, with $\ubold$ being a corresponding eigenfunction. This also immediately implies the statement in Lemma \ref{lem:FriedDtN}(c).
\end{proof}

The following result gives an easy algorithm for computing the nnNtD map, simultaneously for all $\ip$'s, in an arbitrary basis on $\Gamma$, and is adapted from the scalar analogue in \cite{LeMa}.

\begin{lemma} Let $\{f_l \}_{l=1}^\infty$ be an arbitrary basis in $\Hbold^{-1/2}(\Gamma)$. Then the matrix elements 
\[
(\NtD_{\ip})_{l,l'}:=\Scal{\NtD_\ip f_l, f_{l'}}{L^2(\Gamma)}
\]
of the nnNtD map in this basis are given by
\[
(\NtD_{\ip})_{l,l'}=\sum_{j=1}^\infty \frac{1}{\Lambda_{\Neu,j}-\ip} \Scal{f_l, \Pn\Ubold_j|_\Gamma}{L^2(\Gamma)}\Scal{\Pn\Ubold_j|_\Gamma,f_{l'}}{L^2(\Gamma)},
\]
where $\Ubold_j(\xbold)$ are the eigenvectors of the Neumann problem \eqref{eq:elasticNeumann} corresponding to the eigenvalues $\Lambda_{\Neu,j}$ and orthonormalised by
\[
\Scal{\Ubold_j,\Ubold_{j'}}{\Lbold^2_{\rho_s}(\Omega)}=\delta_{j,j'}.
\]
\end{lemma}

\begin{proof} Using Green's formula \eqref{eq:Lbyparts}, we have
\begin{equation}\label{eq:Lbyparts1}
(\NtD_{\ip})_{l,l'}=E\left(\ubold_{l,\ip},\ubold_{l',\ip}\right)-\Lambda\Scal{\ubold_{l,\ip},\ubold_{l',\ip}}{\Lbold^2_{\rho_s}(\Omega)},
\end{equation}
where $\ubold=\ubold_{l,\ip}(\xbold)$ is the solution of \eqref{eq:elasticNeunonhom} with $f=f_l$. We now use the fact that the set of Neumann eigenfunctions $\{\Ubold_j\}$ is a basis in $\Hbold^1(\Omega)$, and we can therefore expand each $\ubold_{l,\ip}$ as
\[
\ubold_{l,\ip}(\xbold)=\sum_{j=1}^\infty \Scal{\ubold_{l,\ip},\Ubold_j}{\Lbold^2_{\rho_s}(\Omega)} \Ubold_j(\xbold).
\]
Substituting this into \eqref{eq:Lbyparts1} and using Green's formula once more and the normalisation condition gives
\[
(\NtD_{\ip})_{l,l'}=\sum_{j=1}^\infty\left(\Lambda_{\Neu,j}-\ip\right)\Scal{\ubold_{l,\ip},\Ubold_j}{\Lbold^2_{\rho_s}(\Omega)}\Scal{\Ubold_j,\ubold_{l',\ip}}{\Lbold^2_{\rho_s}(\Omega)}.
\]
Another integration by parts gives
\[
\Scal{\ubold_{l,\ip},\Ubold_j}{\Lbold^2_{\rho_s}(\Omega)}=\frac{1}{\Lambda_{\Neu,j}-\ip}\Scal{f_l, \Pn\Ubold_j|_\Gamma}{L^2(\Gamma)},
\]
and the result follows.
\end{proof}

\subsection{{Dependence of the eigenvalues of the nnDtN map on the Lam\'{e} parameters}}

{Similarly to what we have done the previous sections, we can consider a family of normal-normal Dirichlet-to-Neumann maps for a fixed $\Lambda$  but varying the Lam\'{e} coefficients $\lambda$, $\mu$. This is of interest for the inverse problem, since shifts in measured eigenvalues could be correlated with changes in elasticity constants. The behaviour of eigenvalues of this family is the same as before, the only difference being that they are monotone \emph{increasing} functions of $\lambda$ and $\mu$ on the intervals of continuity (and some of them blow up at the values $\lambda_0$, $\mu_0$ for which $\Lambda$ becomes a Dirichlet eigenvalue  of $\mathcal{L}(\Lambda)$).  This follows from exactly the same argument, the only difference being that the form \eqref{eq:DtNminmax} is monotone increasing in $\lambda$ and $\mu$. As an illustration, we show plots of  some eigenvalues of the nnDtN map for the unit disk  as  functions of either $\lambda$ or $\mu$ in Figure \ref{fig:figlambdamu}.}

\begin{figure}[!hbt]
\begin{center}
\includegraphics[width=0.95\textwidth]{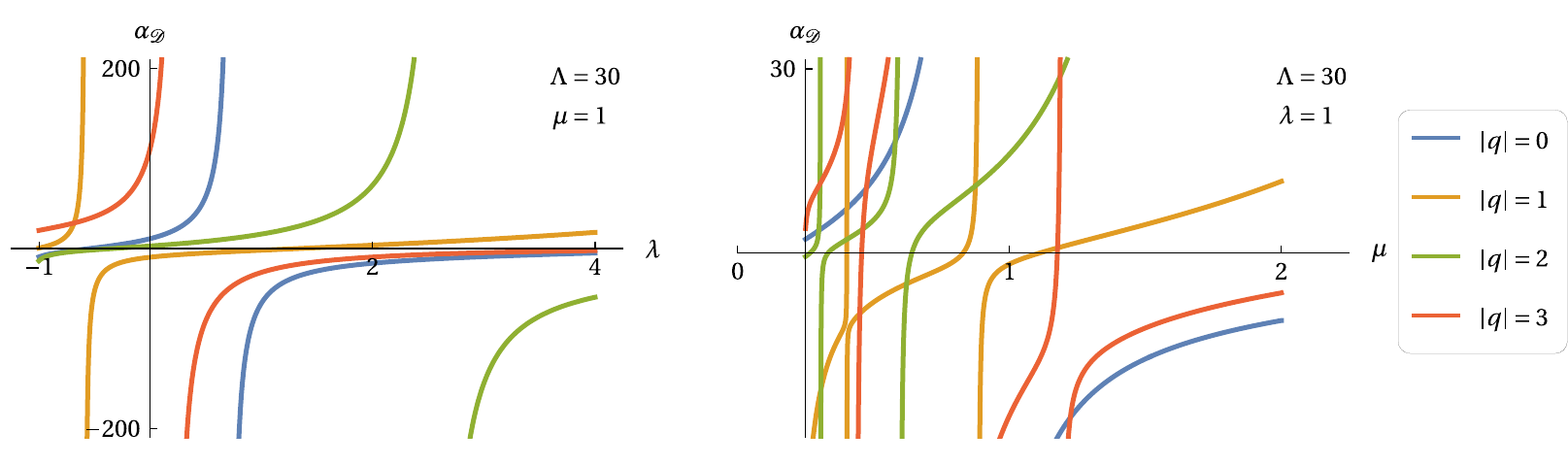}
\caption{Some eigenvalues $\alpha_{\DtN_{\ip}}$ of the nnDtN map for a unit disk as functions of $\lambda$ (left figure) and $\mu$ (right figure) for fixed $\Lambda=30$.\label{fig:figlambdamu}}
\end{center}
\end{figure}
 
 \section{The forward fluid-solid interaction problem}\label{sec:directFSproblem}
We now consider the fluid-solid interaction problem.  We will apply the previous theory when we propose
our target signatures for this problem. As before we consider a bounded elastic body, but now immersed in a compressible, inviscid fluid occupying the exterior domain
\[
\Omega_e := \mathbb{R}^m \backslash \overline{ \Omega}. 
\]
When an incident acoustic wave in the fluid strikes the body, part of its energy is transmitted into the structure in the form of (small) vibrations; in turn, the structure's vibrations produce acoustic waves in the fluid. The forward fluid-solid interaction problem consists of determining the response of the system under the assumption that the geometry and properties of the solid and the fluid, as well as the incident sound wave, are known.  We now recall a standard mathematical formulation of this problem~\cite{hsi00} where we assume that  that the wave is of small amplitude, and the fluid and target motions are time-harmonic, so that we may work in the frequency domain.

We denote the wavenumber in the fluid by 
 \[
k := \frac{\omega}{c_f}\in \mathbb{R},
\] 
where $\omega $ is the angular frequency of the fields and $c_f $ is the speed of sound in the fluid ($\omega$ and $c_f$ are positive real constants). In the fluid the density is a constant $\rho_f \in \mathbb{R}$ such that $\rho _f >0$.

Let $p^i$ represent the incident wave, which must  be a  smooth solution of the Helmholtz equation in a neighbourhood of
$\Omega$ for the given wavenumber $k$.
In this paper $p^i$ is usually chosen to be a propagating plane wave:
\begin{equation}
p^i(\xbold)= p^i(\xbold,\dbold) := {\mathrm{e}}^{\imath k\xbold\cdot \dbold} \qquad\mbox{in }\mathbb{R}^m \, ,\label{incident}
\end{equation}
where $\dbold\in\mathbb{S}^{m-1} =\{\xbold\in\mathbb{R}^m \, ; \,\, |\xbold|=1 \} $ is the direction of propagation. Note that any plane wave satisfies the Helmholtz equation in the whole space $\mathbb{R}^m$.

Under the previous hypotheses, the elastodynamic displacement field $\ubold \in \Hone$ and the dynamic component of the fluid pressure $p\in H_{\mathrm{loc}}^1( \Omega_e)$  solve the system \begin{subnumcases}{\label{eq:PM}}
\nabla \cdot \sigma (\ubold) + \rho_s \Lambda \ubold = \zerobold &in  $\Omega$,\label{eq:PMa}\\
\displaystyle \Delta p + k^2 p = 0 &in $\Omega_e$,\label{eq:PMb}\\
\Tbold\ubold = - p \,\nbold &on  $\Gamma$,\label{eq:PMc}\\
\ubold \cdot \nbold =\frac{1}{\rho_f\Lambda} \frac{\partial p}{\partial \nbold} &on $\Gamma$,\label{eq:PMd}\\
p = p^i + p^s,\quad\text{with}\quad\frac{\partial p^s}{\partial r}  - \imath k p^s = \mathrm{o}\left( r^{-(m-1)/2}\right)&as $r\to\infty$.\label{eq:PMe}
\end{subnumcases}

Here \eqref{eq:PMb} is the  acoustic equations in the time-harmonic regime, respectively. The transmission conditions \eqref{eq:PMc} and \eqref{eq:PMd} on $\Gamma$ are the  dynamic and kinematic boundary conditions, which represent the equilibrium of forces and the equality of the normal displacements (of the solid and the fluid), respectively
\cite{hsi00}. Equation \eqref{eq:PMe} states that the fluid pressure $p$  is the superposition of the given incident field $p^i$ and an unknown scattered wave $p^s$. The latter is selected to be an outgoing wave by the decay condition at infinity: this must hold uniformly in all directions $\hat{\xbold }:=\xbold /r$ when $r = |\xbold | \to \infty$ and is known as the \emph{Sommerfeld radiation condition}. For a complete description of the derivation of this model see \cite{hsi00} or \cite[\S 2]{luke+martin} and the references therein.

Under our assumptions, it is well known that  \eqref{eq:PM} has at most one solution $p$; however, to prove existence of $p$ and $\ubold$, we need to assume additionally that $\omega$ is not a  \emph{Jones frequency} for the solid. The definition of Jones frequencies can be found at the end of \S\ref{subs:bvpelasticity}, and throughout the remainder of this paper we shall assume that $\omega$ is \emph{not} a Jones frequency for $\Omega$. 

It is useful to note that we can immediately reduce \eqref{eq:PM} to the  exterior domain only, with the solid behaviour incorporated via the operator $\NtD_{\ip}$ for $\ip=\omega^2$. Namely, comparing equations \eqref{eq:PMa} and \eqref{eq:PMc},  re-written as 
\begin{equation}\label{eq:elasticwithp}
\begin{cases}
\Lcurlybold(\Lambda)\ubold=\zerobold&\quad\text{in }\Omega,\\
\Pbold_\tang \Tbold\ubold=\zerobold&\quad\text{on }\Gamma,\\
\Pn\Tbold\ubold=-p&\quad\text{on }\Gamma, 
\end{cases}
\end{equation}
with the definition of the nnNtD map, we obtain
\[
\Pn (\ubold|_\Gamma)=\NtD_{\Lambda}\left(\Pn\Tbold\ubold\right)=\NtD_{\Lambda}(-p). 
\]
Substituting this into the remaining equations of \eqref{eq:PM} gives
\begin{subnumcases}{\label{eq:PM1}}
\Delta p + k^2 p = 0 &in $\Omega_e$,\\[.5ex]
\frac{\partial p}{\partial \nbold}+\rho_f\Lambda \NtD_{\Lambda}(p)=0&on $\Gamma$,\label{PM1b}\\[.5ex]
p = p^i + p^s,\quad\text{with}\quad\frac{\partial p^s}{\partial r}  - \imath k p^s = \mathrm{o}\left( r^{-(m-1)/2}\right)&as $r\to\infty$.
\end{subnumcases}

\subsection{The far field operator}\label{subsec:ffoperator}

Since $p^s$ is a radiating solution of the Helmholtz equation, it admits the asymptotic expansion (c.f. \cite{col19})
\[
p^s (\xbold) \, = \, \frac{\mathrm{e}^{\imath k r}}{r^{(m-1)/2}} \,\left(p^{\infty} (\hat{\xbold}
) + \mathrm{O} \!\left(\frac{1}{r^{(m+1)/2}}\right)\right) \quad \mbox{as } r\to\infty \, .
\]
The function $p^{\infty}$ is called the \emph{far field pattern} of the scattered field and can be written in terms of the scattered field in integral form:
\[
p^{\infty} (\hat{\xbold}
) = - c_m\, \int\limits_{\tilde \Gamma} \left( \imath k p^s (\ybold) \nbold(\ybold) \cdot \hat{\xbold} + \frac{\partial p^s}{\partial \nbold} (\ybold) \right) \, {\mathrm{e}}^{-\imath k\ybold\cdot \hat{\xbold}} \,dS_{\ybold} \, ,
\]
see \cite[Eqs. (4.5-6)]{cakoni+colton} or \cite[Eqs. (6-7)]{elschner+hsiao+rathsfeld}. In the expression above, $\tilde \Gamma  = \partial \tilde \Omega$ is the boundary of any bounded regular domain $\tilde \Omega \subseteq \mathbb{R}^m$ that contains $\Omega$ (possibly $\tilde \Omega = \Omega$), and 
\[
c_m :=
\begin{cases}
\frac{\mathrm{e}^{\imath \pi/4}}{\sqrt{8\pi k}} & \quad\text{if $m=2$,}\\[1ex]
\frac{1}{4\pi } & \quad \text{if $m=3$}.
\end{cases}
\]
Note that the standard Rellich's Lemma holds in the fluid so guaranteeing uniqueness of the pressure field there  (see e.g. \cite[Th. 4.1]{cakoni+colton}) as stated next.

\begin{lemma}\label{cor_ffid} 
If $p^{\infty} (\hat{\xbold}) = 0$ for all $\hat{\xbold} \in \mathbb{S} ^{m-1}$, then $p^s (\xbold) = 0$ for $\xbold\in\Omega_e$.
\end{lemma}

Let us consider an incident plane wave $p^i=p^i(\cdot,\dbold)$  as in (\ref{incident}), 
and denote by $\ubold (\cdot ,\dbold )$, $p^s(\cdot , \dbold )$ and $p^{\infty}(\cdot , \dbold) $ the associated displacement field, scattered wave and far field pattern, respectively. Because of the presence of the solid, it is less obvious that the following \emph{reciprocity relation} 
holds:
\begin{equation}\label{eq:reciprocity}
p^{\infty} (-\hat{\xbold}, \dbold) = p^{\infty} (-\dbold,\hat{\xbold}) \qquad \mbox{a.e. }\hat{\xbold}, \dbold\in\mathbb{S}^{m-1},
\end{equation}
but this is indeed the case, cf. \cite[Lemma 2.2]{monk+selgas_2} and references therein.

The \emph{far field operator} $F: L^2 (\mathbb{S}^{m-1}) \to L^2 (\mathbb{S}^{m-1})$ is then defined by
\[
(Fg) (\hat{\xbold} ) := \int\limits_{\mathbb{S}^{m-1}} p^{\infty} (\hat{\xbold}, \dbold) \, g (\dbold) \, d S_{\dbold} \qquad
\text{a.e. }\hat{\xbold}\in\mathbb{S}^{m-1}.
\]
Notice that, by the linearity of the forward problem \eqref{eq:PM}, $p^{\infty}_g:=Fg$ is the far field pattern of the wave scattered by the incident field  $ p^i= p^i_g $, where
\begin{equation}\label{defHerglotzfunc}
p^i_g (\xbold) :=
\int\limits_{\mathbb{S}^{m-1}} \; {\mathrm{e}}^{\imath k \xbold \cdot \ybold} \, g (\ybold) \, dS_{\ybold} =  \displaystyle\int\limits_{\mathbb{S}^{m-1}} \; p^i( \xbold, \ybold) \, g (\ybold) \, dS_{\ybold} 
\end{equation}
is the \emph{Herglotz wave function} with kernel $g\in L^2 (\mathbb{S}^{m-1})$~\cite{col19}.
The far field operator $F:L^2 (\mathbb{S}^{m-1}) \to L^2 (\mathbb{S}^{m-1})$ is injective and has dense range if, and only if, $\omega $ is not an \emph{interior transmission eigenvalue} associated to the fluid-solid interaction problem and with an eigenfunction of the form of a Herglotz wave function (see  \cite[Lemmas 2.3, 2.4]{monk+selgas_1}).  In the next section, we shall
connect the injectivity of a modified far field operator to a class of nnNtD interior eigenvalues for the solid.

\section{Impedance type modification of the far field operator}\label{sec:ffoperator_modification1} 
From now on $\omega$ (and so $\Lambda$) is a fixed non-zero real parameter (the angular frequency of the field), and it is assumed not  to be a Jones frequency.
We now introduce a modified far field operator which makes use of the following auxiliary problem:
\begin{subnumcases}{\label{PM_Steklov}}
\displaystyle \Delta h + k^2 h = 0 &in $\Omega_e$,\\[.5ex]
\displaystyle \frac{\partial h}{\partial \nbold}+\alpha h=0&on $\Gamma$,\label{PM_Steklov_b}\\[.5ex]
\displaystyle h = h^i + h^s,\quad\text{with}\quad\frac{\partial h^s}{\partial r}  - \imath k h^s = \mathrm{o}\left( r^{-(m-1)/2}\right)&as $r\to\infty$,
\end{subnumcases}
where $\alpha\in \mathbb{C}$ with $\Im(\alpha)\geq 0$ is a fixed parameter.
Using Rellich's lemma and the Fredholm alternative, this auxiliary problem is well-posed, and its solution belongs to $\mathcal{C}^2(\Omega_e)\cap\mathcal{C}^1(\overline{\Omega_e})$ provided both $\Omega$ and $h^i$ are smooth enough; see the comments below problem (1.2) in \cite{cakoni+al2016}. Let us consider an incident plane wave $h^i=p^i(\cdot,\dbold)$  as in (\ref{incident}), and denote the corresponding scattered wave and its far field pattern by  $h^s(\cdot , \dbold ) $ and $h^{\infty}(\cdot , \dbold) $, respectively. We introduce the associated far field operator 
\[
(F_{\alpha}g) (\hat{\xbold} ) := \int\limits_{\mathbb{S}^{m-1}} h^{\infty} (\hat{\xbold}, \dbold) \, g (\dbold) \, dS_{\dbold} \qquad
\text{a.e. }\hat{\xbold}\in\mathbb{S}^{m-1}.
\]
We use it to define the impedance type \emph{modified far field operator} $\mathcal{F}_{\alpha}:=F-F_{\alpha}$, so that 
\begin{equation}\label{ffop_modified}
(\mathcal{F}_{{\alpha}}g) (\hat{\xbold} ) = \int\limits_{\mathbb{S}^{m-1}} \big( p^{\infty} (\hat{\xbold}, \dbold) - h^{\infty} (\hat{\xbold}, \dbold) \big) \, g (\dbold) \, dS_{\dbold} \qquad
\text{a.e. }\hat{\xbold}\in\mathbb{S}^{m-1}.
\end{equation}
\subsection{Rescaled impedance eigenvalues}\label{subsec:eigsSteklov}

In order to link the modified far field operator to certain rescaled impedance eigenvalues, let us recall that the usual interior transmission eigenvalues arise in the analysis of the injectivity of the far field operator. Accordingly, we study formally the injectivity of the modified far field operator 
$\mathcal{F}_{{\alpha}}: L^2 (\mathbb{S}^{m-1}) \to L^2 (\mathbb{S}^{m-1})$. To this end, we consider $g\in L^2 (\mathbb{S}^{m-1})$ such that $\mathcal{F}_{{\alpha}}g =0$ in $\mathbb{S}^{m-1}$, that is,
\begin{equation}\label{inj-modF_rewritten}
 \int\limits_{\mathbb{S}^{m-1}} p^{\infty} (\dbold, \ybold)  \, g (\ybold) \, dS_{\ybold}
=
 \int\limits_{\mathbb{S}^{m-1}} h^{\infty} (\dbold, \ybold) \, g (\ybold)  \, dS_{\ybold} \qquad
\text{a.e. }\dbold\in\mathbb{S}^{m-1} \, .
\end{equation}
We can rewrite this condition by considering 
the 
incident fields $p^i=p^i_g$ and $h^i=h^i_g$, where $p^i_g=h^i_g$ is the Herglotz wave function with density $g$ defined in (\ref{defHerglotzfunc}). Indeed, if $p^{\infty}_g$ and $h^{\infty}_g$ denote the  far field patterns of the associated scattered fields, then (\ref{inj-modF_rewritten}) implies that $p^{\infty}_g = h^{\infty}_g$ in $\mathbb{S}^{m-1}$.
By Rellich's lemma, this implies that the scattered fields $p^s_g$ and $h^s_g$ match in  $\overline{\Omega_e}$; hence, the boundary condition satisfied by $h_g$ in (\ref{PM_Steklov}) implies that 
\[
\frac{\partial p_g}{\partial \nbold}+\alpha p_g=0 \qquad \text{on } \Gamma  \, .
\]
Thanks to the transmission conditions of problems \eqref{eq:elasticwithp} and \eqref{eq:PM1}, the above is equivalently written in terms of the associated solid displacements $\ubold_g$ as
\[
\rho_f\Lambda \, \Pn (\ubold_g|_{\Gamma}) \, \nbold - \alpha \,\Tbold\ubold_g= 0 
\qquad \text{on } \Gamma,
\]
or as
\begin{equation}\label{eq:rescaled}
\NtD_{\Lambda}\Pn\Tbold\ubold_g =\frac{\alpha}{\rho_f\Lambda} \Pn\Tbold\ubold_g .
\end{equation}
Here again, we denote $\ip=\omega^2$. We note that for given $\omega\ne 0$ and $\alpha\in\mathbb{C}$, problem \eqref{eq:rescaled} has a non-trivial solution $\ubold_g\ne\zerobold$ if, and only if,
\[
\frac{\alpha}{\rho_f\Lambda}\in\Spec(\NtD_{\Lambda}).
\]
This discussion suggests the following definition.

\begin{definition}[Rescaled impedance eigenvalues]\label{defn:rie} 
For a given $\omega\ne 0$, we will call the elements of the multiset (with multiplicities)
\[
\tilde{\Spec}\left(\NtD_{\omega^2 }\right):=\left\{\frac{\alpha}{\rho_f\omega^2};\, \alpha\in\Spec\left(\NtD_{\omega^2}\right)\right\}
\]
the \emph{rescaled impedance eigenvalues} of $\Omega$.
\end{definition}
With this definition we can now formally state the target signature strategy related to impedance eigenvalues.

\noindent{\bf Impedance target signatures. }\label{defn:rie2} 
\emph{Given far field data $p^{\infty} (\hat{\xbold}, \dbold)$ for all $\hat{\xbold}, \dbold\in \mathbb{S}^{m-1}$, determine the rescaled impedance eigenvalues of $\Omega$.  This discrete set of eigenvalues is the proposed \emph{impedance target signatures}.}

\begin{remark} In practice we would not have data for all $(\hat{\xbold}, \dbold)$ but only noisy measurements for a finite number of pairs.  Thus we cannot hope to determine all the rescaled impedance eigenvalues.  Furthermore, the number that can be determined is limited by noise on the data. For more information see Section~\ref{sec:NumSteklovFS}.
\end{remark}
\subsection{Determination of impedance eigenvalues from far field data}\label{subsec:SteklovFromFFP}
We next study the determination of impedance eigenvalues from far field data by using the modified far field operator $\mathcal{F}_{{\alpha}}=F-F_{\alpha}$ defined in (\ref{ffop_modified}). More precisely, following \cite{cogar+al2017b} we propose to find (approximate) solutions $g_{\zbold}\in L^2(\mathbb{S}^{m-1})$ of the modified far field equation 
\begin{equation}\label{eq:ffeq}
\mathcal{F}_{{\alpha}} g_{\zbold}=\Phi^{\infty}_{\zbold} \quad \text{for points } \zbold \in \Omega .
\end{equation}
Here and in the sequel 
\[
\Phi^{\infty}_{\zbold} (\xbold)
= \begin{cases}
\displaystyle\frac{\mathrm{e}^{\mathrm{i}\pi/4} }{\sqrt{8\pi k } }  e^{-\mathrm{i}k \,\hat{\xbold}\cdot\zbold} \qquad & \text{if } m=2,\\[1ex]
\displaystyle\frac{1}{4 \pi } \mathrm{e}^{-\mathrm{i}k \,\hat{\xbold}\cdot\zbold } & \text{if } m=3,
\end{cases}
\]
is the far field pattern of a point source located at $\zbold$ in a purely fluid domain: 
\[
\Phi_{\zbold} (\xbold) = 
\begin{cases}
\displaystyle\frac{\mathrm{i}}{4}  H^{(1)}_0(k |\xbold-\zbold|) \qquad & \text{if }m=2,\\[1ex]
\displaystyle\frac{e^{\mathrm{i}k |\xbold-\zbold|}}{4 \pi |\xbold-\zbold| } & \text{if }m=3.
\end{cases}
\]
We expect that the norm of such approximate solutions will blow up whenever $\alpha$ is a rescaled impedance eigenvalue. Before justifying this approach, we note that to solve (approximately) these modified far field equations, we need $\mathcal{F}_{{\alpha}}: L^2(\mathbb{S}^{m-1})\to L^2(\mathbb{S}^{m-1})$ to be injective and to have dense range. 

\begin{lemma}\label{lem:F1to1}
The modified far field operator  $\mathcal{F}_{{\alpha}}: L^2(\mathbb{S}^{m-1}) \to L^2(\mathbb{S}^{m-1})$ in equation (\ref{ffop_modified}) is one-to-one and has dense range if,  and only if,  $\alpha$ is not a rescaled impedance eigenvalue with an eigenmode of the form 
\begin{equation}
\ubold_g=\displaystyle\int\limits_{\mathbb{S}^{m-1}} \ubold(\cdot,\ybold)g(\ybold)\, dS_{\ybold}\;\mbox{ for some }g\in L^2(\mathbb{S}^{m-1}).\label{eq:formofug}
\end{equation}
\end{lemma}

\begin{proof}
Let us start by characterizing when $\mathcal{F}_{{\alpha}}$ is injective. To this end, 
for each $g\in L^2(\mathbb{S}^{m-1})$  
we consider in \eqref{eq:PM} and (\ref{PM_Steklov}) the incident waves given by the Herglotz {wave} function with kernel $g$, that is, $p^i=p^i_g$ and $h^i = h^i_g$. This allows us to rewrite the property $\mathcal{F}_{{\alpha}}g=0$ as $p_g^s=h_g^s$ in $\Omega_e$, in which case we may reason as in \S\ref{subsec:eigsSteklov} to deduce that the displacements field $\ubold_g$ solves \eqref{eq:rescaled}. Thus, $\mathcal{F}_{{\alpha}}$ is injective as long as $\alpha$ is not a rescaled impedance eigenvalue whose eigenmode is of the form given in (\ref{eq:formofug}).

Next, we study when $\mathcal{F}_{{\alpha}}$ has dense range in $L^2(\mathbb{S}^{m-1})$ or, equivalently, when its adjoint operator $\mathcal{F}_{{\alpha}}^*$ is one-to-one. Notice that, for any $g_1,g_2\in L^2(\mathbb{S}^{m-1})$ we have
\begin{align*}
\int\limits_{\mathbb{S}^{m-1}} \mathcal{F}_{{\alpha}}^*g_1(\xbold)\,\overline{g_2}(\xbold)\,dS_{\xbold}&=
\int\limits_{\mathbb{S}^{m-1}} g_1(\xbold)\, \overline{\mathcal{F}_{{\alpha}}g_2 (\xbold)} \,dS_{\xbold}
\\
&=\int\limits_{\mathbb{S}^{m-1}} g_1(\xbold)
\overline{\int\limits_{\mathbb{S}^{m-1}} (p^{\infty}(\xbold,\dbold)-h^{\infty}(\xbold,\dbold)) g_2(\dbold) \,dS_{\dbold} }
\, \,dS_{\xbold} ;
\end{align*}
then, by changing the order of integration and applying the reciprocity relation (\ref{eq:reciprocity}),
\begin{align*}
\int\limits_{\mathbb{S}^{m-1}} \mathcal{F}_{{\alpha}}^*g_1(\xbold)\,\overline{g_2}(\xbold)\,dS_{\xbold}
&=\int\limits_{\mathbb{S}^{m-1}} \overline{g_2(\dbold)}
\int\limits_{\mathbb{S}^{m-1}} \overline{(p^{\infty}(-\dbold,-\xbold)-h^{\infty}(-\dbold,-\xbold))} g_1(\xbold) \,dS_{\xbold} 
\, \,dS_{\dbold} 
\\
&=\int\limits_{\mathbb{S}^{m-1}} \overline{g_2(-\dbold)}
\overline{\int\limits_{\mathbb{S}^{m-1}} (p^{\infty}(\dbold,\xbold)-h^{\infty}(\dbold,\xbold)) \, \overline{g_1(-\xbold)} \,dS_{\xbold} }
\, \,dS_{\dbold}\\
&=\int\limits_{\mathbb{S}^{m-1}} {\tilde{g}_2(\dbold)} \, \overline{\mathcal{F}_{{\alpha}}\tilde{g}_1(\dbold)} \,dS_{\dbold} ,
\end{align*}
where we denote $\tilde{g}_j(\xbold)=\overline{g_j(-\xbold)}$ in $\mathbb{S}^{m-1}$ for $j=1,2$. Thus, $\mathcal{F}_{{\alpha}}^*$ is one-to-one if and only if $\mathcal{F}_{{\alpha}}$ is one-to-one. This means that 
 $\mathcal{F}_{{\alpha}}: L^2(\mathbb{S}^{m-1})\to L^2(\mathbb{S}^{m-1})$ has dense range if, and only if, it is one-to-one. This completes the proof.
 \end{proof}

Now we study the behaviour of (approximate) solutions $g_{\zbold}\in L^2(\mathbb{S}^{m-1})$ of the modified far field equation $\mathcal{F}_{{\alpha}} g_{\zbold}=\Phi^{\infty}_{\zbold}$ for points $\zbold$ inside the target $\Omega$, depending on the parameter $\alpha\in\mathbb{R}$.  In general our approach follows \cite{cakoni+al2016}.

\subsection{Behaviour when $\alpha$ is not a rescaled impedance eigenvalue}\label{subsubsec:notSteklov}
Let us consider any point $\zbold\in\Omega $ and try to build an approximate solution  $g\in L^2(\mathbb{S}^{m-1})$ of the modified far field equation
\[
\mathcal{F}_{{\alpha}}g 
\, = \, \Phi^{\infty}_{\zbold}\quad \text{a.e. in } \mathbb{S}^{m-1} .
\]
Due to to having $\zbold\in\Omega $, we can apply Rellich's lemma to deduce that this equation is fulfilled if, and only if,
\[
 p^{s}_g-h^{s}_g \, = \, \Phi_{\zbold} \quad \text{a.e. in } \Omega_e .
\]
Then, using the boundary condition for the field $h_g$ in \eqref{PM_Steklov},
\[
\displaystyle\frac{\partial p_g}{\partial \nbold}+\alpha p_g=
\displaystyle\frac{\partial \Phi_{\zbold}}{\partial \nbold}+\alpha  \Phi_{\zbold} \qquad \mbox{on } \Gamma .
\]
Equivalently,  in terms of the displacements field $\ubold_g$, we have, cf. \eqref{eq:elasticwithp}, \eqref{eq:PM1}, and \eqref{eq:rescaled},
\begin{equation}\label{eq:nonhomSt}
\NtD(f)_{\Lambda}-\frac{\alpha}{\rho_f\Lambda}f=
\frac{1}{\rho_f\Lambda}\left(\frac{\partial \Phi_{\zbold}}{\partial \nbold}+\alpha  \Phi_{\zbold}\right) \qquad \text{on } \Gamma ,
\end{equation}
where $f:=\Pn\Tbold \ubold$ and $\Lambda=\omega^2$. 

Assuming that $\alpha$ is not a rescaled impedance eigenvalue, the problem above is well-posed and has the unique solution 
\[
f=\frac{1}{\rho_f\Lambda}\left(\NtD_{\Lambda}-\frac{\alpha}{\rho_f\Lambda}\right)^{-1}\left(\frac{\partial \Phi_{\zbold}}{\partial \nbold}+\alpha  \Phi_{\zbold}\right)\in H^{-1/2}(\Gamma),
\] 
and we can therefore recover a unique  $\ubold\in \Hone$ by solving \eqref{eq:elasticmixednonhom} for $\Lambda=\omega^2$; however, this unique solution is not  necessarily of the form $\ubold=\ubold_g$ as in \eqref{eq:formofug}.
 
In the fluid domain, we are looking for $p_g\in H^1_{\mathrm{loc}}(\Omega_e)$ such that
\begin{equation}\label{MP_fluid_z}
\begin{cases}
\Delta p_g + k^2 p_g= 0 & \qquad \text{in }\Omega_e,\\[.5ex]
p _g = -\Tbold\ubold \cdot\nbold& \qquad \text{on } \Gamma,\\[.5ex]
\displaystyle \frac{\partial p_g}{\partial \nbold}  = \rho_f\Lambda \ubold \cdot \nbold  & \qquad \text{on } \Gamma,\\[.5ex]
\displaystyle p_g = p_g^i + p_g^s,\quad \text{ with }\frac{\partial p_g^s}{\partial r}  - \imath k p_g^s = \mathrm{o}\left( r^{-(m-1)/2}\right) &\qquad\text{ as }r\to\infty.
\end{cases}
\end{equation}
 At first glance, this problem seems to have too many constraints. However, this is not the case because $p^i_g$ is not given above. In consequence, we search for  $p^i$ not in the form of a Herglotz function $p^i=p^i_g$  
with density $g\in L^2(\mathbb{S}^{m-1})$ but in the larger space of solutions of Helmholtz equation in $\Omega$:
\[
\mathbb{H}_{\mathrm{inc}}(\Omega ) = \{ q\in H^1(\Omega); \, \Delta q+k^2 q= 0 \mbox{ in }\Omega \} .
\]
In other words, we look for $p^i \in H^1(\Omega)$ and $p^s\in H^1_{\mathrm{loc}}(\Omega_e)$ such that
\begin{equation}\label{MP_fluid_z_aux}
\begin{cases}
\Delta p^i + k^2 p^i= 0 & \qquad \text{in }\Omega,\\[.5ex]
\Delta p^s + k^2 p^s= 0 & \qquad \text{in }\Omega_e,\\[.5ex]
 p^i+p^s = -\Tbold\ubold \cdot\nbold & \qquad \text{on } \Gamma,\\[.5ex]
\displaystyle\frac{\partial p^i}{\partial \nbold}+ \displaystyle\frac{\partial p^s}{\partial \nbold} = \rho_f\Lambda \ubold \cdot \nbold & \qquad \text{on } \Gamma,\\[1.25ex]
\displaystyle\frac{\partial p^s}{\partial r}  - \imath k p^s = \mathrm{o}\left( r^{-(m-1)/2}\right) & \qquad\text{ as }r\to\infty. 
\end{cases}
\end{equation}
This transmission problem is well-posed, although here again its unique solution is not necessarily of the form 
$p^i=p^i_g$ and $p^s=p^s_g$ with $g\in L^2(\mathbb{S}^{m-1})$.

Summing up, we have seen that, when $\alpha\in\mathbb{R}$ is not a rescaled impedance eigenvalue, for any $\zbold\in\Omega$ there exists some incident field $p^i\in\mathbb{H}_{\mathrm{inc}}(\Omega)$ for which $p^{\infty}$ and $h^{\infty}$ (the far field patterns  of the scattered fields $p^s$ and $h^s$ that solve \eqref{eq:PM} and (\ref{PM_Steklov}), respectively) satisfy
$$
p^{\infty}-h^{\infty}=\Phi_{\zbold}^{\infty} \, .
$$ 
This gives us an approximate solution of the modified far-field equation $\mathcal{F}_{{\alpha}}g=\Phi^{\infty}_{\zbold}$ in $\mathbb{S}^{m-1}$ by approximating the incident field $p^i\in\mathbb{H}_{\mathrm{inc}}(\Omega)$ with a Herglotz wave function $p^i_g$. 

More precisely, in the usual way we may 
factorise  $\mathcal{F}_{{\alpha}}$ in terms of the following well defined and bounded operators:
\begin{itemize}
\item $\mathcal{H}: L^2(\mathbb{S}^{{m-1}})\to\mathbb{H}_{\mathrm{inc}}(\Omega)$ maps each function $g$ into the associated Herglotz wave function $\mathcal{H}g=p^i_g$;
\item $\mathcal{G}: \mathbb{H}_{\mathrm{inc}}(\Omega)\to L^2(\mathbb{S}^{{m-1}})$ maps any incident wave into  the far field pattern $\mathcal{G}p^i=p^{\infty}$ of the scattered wave $p^s$ that solves the fluid-solid interaction problem \eqref{eq:PM} for the incident field $p^i$;
\item $\mathcal{G}_{\alpha}: \mathbb{H}_{\mathrm{inc}}(\Omega)\to L^2(\mathbb{S}^{{m-1}})$ maps each incident field $h^i$ into the far field pattern of $h^s$ the solution of the auxiliary problem (\ref{PM_Steklov}).
\end{itemize}
We have shown that 
 $\Phi^{\infty}_{\zbold}$ is in the range of $(\mathcal{G}-\mathcal{G}_{\alpha})$. Recall that the range of $\mathcal{G}-\mathcal{G}_{\alpha}$ is dense in that of $\mathcal{F}_{{\alpha}}$, indeed we have $\mathcal{F}_{{\alpha}}=(\mathcal{G}-\mathcal{G}_{\alpha})\circ\mathcal{H} :  L^2(\mathbb{S}^{{n-1}})\to L^2(\mathbb{S}{^{m-1}})$.  We summarise these results in Theorem~\ref{find_stek} at the end of the next section.

\subsection{Behaviour when $\alpha$ is a rescaled impedance eigenvalue}\label{subsubsec:yesSteklov}
By Definition \ref{defn:rie}, cf. also discussion after \eqref{eq:rescaled}, $\alpha\in\mathbb{R}$  being a rescaled impedance eigenvalue means that there exists  
$\ubold_{\alpha}\in \Hone$,  $ \ubold_{\alpha}\neq \zerobold$, such that
\[
\begin{cases}
\Lbold\ubold_{\alpha} + \rho_s \Lambda \ubold_{\alpha} = \zerobold & \qquad \text{in } \Omega,\\[.5ex]
\Pbold_\tang \ubold_\alpha=\zerobold& \qquad  \text{on } \Gamma,\\[.5ex]
\Pn\ubold_{\alpha}=\displaystyle\frac{\alpha}{\rho_f\Lambda} \Pn\Tbold\ubold_{\alpha}& \qquad  \text{on } \Gamma,
\end{cases}
\]
where $\Lambda=\omega^2\neq 0$ as usual. Then, we can  reason as in the previous case and show that there exist $p^i_{\alpha}\in\mathbb{H}_{\mathrm{inc}}(\Omega)$ and $p_{\alpha}\in H^1_{\mathrm{loc}}(\Omega_e)$ that satisfy \eqref{MP_fluid_z_aux} (see comments below \eqref{MP_fluid_z}), with $\ubold_\alpha$ substituted for $\ubold$.
 We note that we can build {$p_\alpha$ and $p_\alpha^i$} by solving the following transmission problem for {$\tilde{p}_\alpha$}, where the superscripts $\pm$ indicate whether  the trace or the normal derivative is taken either from the interior of $\Omega$ or from its exterior $\Omega_e$:
\begin{equation}\label{MP_fluid_alpha_aux}
\begin{cases}
\Delta \tilde{p}_{\alpha} + k^2 \tilde{p}_{\alpha}= 0 & \qquad \text{in } \Omega\cup\Omega_e , \\[.5ex]
 \tilde{p}_{\alpha}^+ - \tilde{p}_{\alpha}^- = -\Tbold\ubold_{\alpha} \cdot\nbold& \qquad \text{on }\Gamma ,\\[.5ex]
\displaystyle\frac{\partial \tilde{p}_{\alpha}^+}{\partial \nbold} - \frac{\partial \tilde{p}_{\alpha}^-}{\partial \nbold} = \rho_f\Lambda \ubold_{\alpha} \cdot \nbold  & \qquad \text{on } \Gamma, \\[1.25ex]
\displaystyle\frac{\partial \tilde{p}_{\alpha}}{\partial r}  - \imath k \tilde{p}_{\alpha} = \mathrm{o}\left( r^{-(m-1)/2}\right)& \qquad\text{ as }r\to\infty. 
\end{cases}
\end{equation}
Indeed, we then take  $p^i_{\alpha}=-\tilde{p}_{\alpha}|_{\Omega}$ and $p^s_{\alpha}=\tilde{p}_{\alpha}|_{\Omega_e}$, see (\ref{MP_fluid_z_aux}).

Let us consider any $\zbold\in\Omega$ for which there exists $p^i\in\mathbb{H}_{\mathrm{inc}}(\Omega)$ such that  $(\mathcal{G}-\mathcal{G}_{\alpha})p^i=\Phi^{\infty}_{\zbold}$.  Denoting by  $(\ubold, p) \in \Hone \times H^1_{\mathrm{loc}}(\Omega_e)$  the associated solution of the fluid-solid interaction problem \eqref{eq:PM} and reasoning as for the previous case,  we arrive at equation \eqref{eq:nonhomSt} with $f=\Pn\Tbold \ubold$.  Since $\alpha$ is now a rescaled impedance eigenvalue, this equation, by Fredholm's alternative for the operator $\NtD_{\Lambda}$, is only solvable if   
\[
\Scal{\frac{\partial \Phi_{\zbold}}{\partial\nbold}+\alpha\Phi_{\zbold}|_\Gamma, \ubold_\alpha\cdot \nbold}{L^2(\Gamma)}=0=\int\limits_{\Gamma} \left(  \frac{\partial \Phi_{\zbold}}{\partial\nbold}+\alpha\Phi_{\zbold}\right) \, (\overline{\ubold_{\alpha}}\cdot\nbold) \, dS_{\xbold}.
\]
Using the transmission conditions in \eqref{MP_fluid_alpha_aux}, we can rewrite it in terms of $\tilde{p}_{\alpha}$ as
\begin{equation}\label{hrep}
\int\limits_{\Gamma} 
\left( \frac{\partial \Phi_{\zbold}}{\partial\nbold}   \left[\overline{\tilde{p}_{\alpha}}\right]_{\Gamma}
-
\Phi_{\zbold} \left[ \frac{\partial \overline{\tilde{p}_{\alpha}}}{\partial\nbold}\right]_{\Gamma}
 \right) \, dS_{\xbold}  = 0 \, ,
\end{equation} 
where $\left[\cdot\right]$ is used to denote  the jump across $\Gamma$, and this is the integral representation of the Helmholtz equation   in $\Omega\cup\Omega^e$ {satisfied} by $\tilde{p}_{\alpha}$, up to complex conjugation and a scaling factor. 

Let us suppose that (\ref{hrep}) holds for a.e. $\zbold\in B$, where $B$ is any subset of $\Omega$ with non-zero measure. Then, $\tilde{p}_{\alpha}=0$ in $B$ and, by Rellich's theorem, $\tilde{p}_{\alpha}$ vanishes in the whole domain $\Omega$; that is, $p^i_{\alpha}=0$ in $\Omega$. Therefore, $\ubold_{\alpha}\in \Hbold^1(\Omega)$ and $p_{\alpha}=p^{s}_{\alpha}\in H^1_{\mathrm{loc}}(\Omega_e)$ solve the forward fluid-solid interaction problem \eqref{eq:PM} with null incident field. Assuming that $\omega$ is not a Jones frequency, 
it follows that $\ubold_{\alpha}=\zerobold$ in $\Omega$, which contradicts that it is a rescaled impedance mode. 

We summarise our results in the following theorem.

\begin{theorem}\label{find_stek}
 When $\alpha\in\mathbb{R}$ is not a rescaled impedance eigenvalue, $\Phi^{\infty}_{\zbold}$ is in the range of $(\mathcal{G}-\mathcal{G}_{\alpha})$ for a.e. $\zbold\in\Omega$. In contrast, if $\alpha\in\mathbb{R}\setminus\{0\}$ is a rescaled impedance eigenvalue and $B\subset\Omega$ has non-zero measure, then it is not possible that $\Phi^{\infty}_{\zbold}$ is in the range of $(\mathcal{G}-\mathcal{G}_{\alpha})$ for a.e. $\zbold\in B$.
\end{theorem}
\begin{remark} If $\zbold\not\in\Omega$, then  $\Phi^{\infty}_{\zbold}$ cannot be in the range of $(\mathcal{G}-\mathcal{G}_{\alpha})$, and this case  provides no useful
information about the impedance eigenvalues. 
\end{remark}
It will be convenient in  the examples in the next section to work not with the parameter $\alpha$ but with the rescaled parameter
\begin{equation}\label{eq:tildealpha}
\tilde{\alpha}:=\frac{\rho_f  \omega^2}{\alpha},
\end{equation}
and the rescaled modified far field operator 
\begin{equation}\label{eq:tildeFdefn}
\tilde{\mathcal{F}}_{\tilde{\alpha}}:=\mathcal{F}_{\rho_f \omega^2/\tilde{\alpha}}.
\end{equation}
In terms of the new parameter, Theorem \ref{find_stek} can be restated as:
\begin{corollary}
$\Phi^{\infty}_{\zbold}$ is in the range of $(\mathcal{G}-\mathcal{G}_{\alpha})$ for a.e. $\zbold\in\Omega$ if, and only if, $\tilde{\alpha}\not\in\Spec\left(\DtN_{\omega^2}\right)$.
\end{corollary}%

\section{Numerical examples and experiments}\label{sec:NumSteklovFS}

{In this section we will be working on the unit disk to use a separation of variables. For illustrative purposes, we will consider two sets of physical and mathematical parameters summarised by
\[
\mathcal{P}:=(\lambda,\mu,\rho_s,\rho_f, c_f, {\omega}),
\] 
namely $\mathcal{P}_{5}=(1,0.5,1,1,1,5)$ and  $\mathcal{P}_{0.5}=(1,0.5,1,1,1,0.5),$
which only differ in the value of $\omega$.

{The explicit formulae used for the computation of impedance eigenvalues are collected in \S\ref{app:A},  and various expressions involving far field operators are in \S\ref{app:B}. For practical applications and in all the plots, the infinite sums in \S\ref{app:B} are truncated to summations over $|q|\le 20$.}

\subsection{Impedance eigenvalues for the disk, and the modified far field equations} 

In Supplementary Materials \S\ref{app:A} we show how to compute the nnDtN eigenvalues for the unit disk.  For the two experiments mentioned above, the results are shown in Table~\ref{tab:tab1} in \S\ref{app:A} where we give the eigenvalues $\alpha_q$ of the nnDtN map $\DtN_\Lambda$ less than $13$ for both sets of parameters above. These values are used to compare with predictions from the far field operators.

In Theorem~\ref{find_stek} we showed that the range of the modified solution operator ${\cal G}-{\cal G}_\alpha$ can be used to characterise the rescaled impedance eigenvalues of the solid.  Following the usual Linear Sampling approach for transmission eigenvalues~\cite{cakoni+al2016}, we use approximate solutions of the far field equation \eqref{eq:ffeq} as an indicator
for the presence of eigenvalues. In particular we consider approximate solutions $g_{\zbold}$ of \eqref{eq:ffeq}, or more precisely of its rescaled version
\[
\tilde{\mathcal{F}}_{\tilde{\alpha}} g_{\zbold}=\Phi^{\infty}_{\zbold},
\]
see the definition of $\tilde{\mathcal{F}}_{\tilde{\alpha}}$ in \eqref{eq:tildeFdefn}.
Of course this problem is ill-posed  but by truncating the series representation for the far field pattern, we obtain a regularised problem which is well posed, and it is this truncated problem what we use in the examples.  We expect $\|g_{\zbold}\|$ to peak at values of $\alpha$ corresponding to eigenvalues for almost every $\zbold$. 
We denote by  
\[
G(\zbold):=\|g_{\zbold}\|^2
\]
the $L^2$ norm of $g_{\zbold}$, see \eqref{eq:gznorm1} for its explicit expression.  We can take
\[
\zbold=-\zeta,\qquad\zeta=|\zbold|\in[0,1] 
\]
without loss of generality (this can be always achieved by a change of coordinates). To visualise the behaviour of $G(\zbold)$ over $\zbold\in\overline{D}$, we plot in Figure \ref{fig:Gnorms} the following two quantities  as functions of $\tilde{\alpha}$:
\[
\|G\|_{L^\infty(D)}=\sup_{\zeta\in[0,1]} G(\zbold) \qquad\mbox{and}\qquad  \|G\|_{L^1(D)}=2\pi\int_0^1 | G(\zbold)|\,\mathrm{d}\zeta.
\] 
As already mentioned, in both cases the infinite summation in \eqref{eq:gznorm1} is replaced by the summation over $|q|\le 20$; the first quantity is evaluated by taking $N=101$ equally 
distributed points $\zeta$ in $[0,1]$ and then randomly perturbing each interior point by a random shift independent and identically distributed in $\left[-\frac{1}{2N}, \frac{1}{2N}\right]$, whereas the second is evaluated using the standard relation  \eqref{eq:besselint}. 

\begin{figure}[hbt!]
\centering
\includegraphics[width=0.45\textwidth]{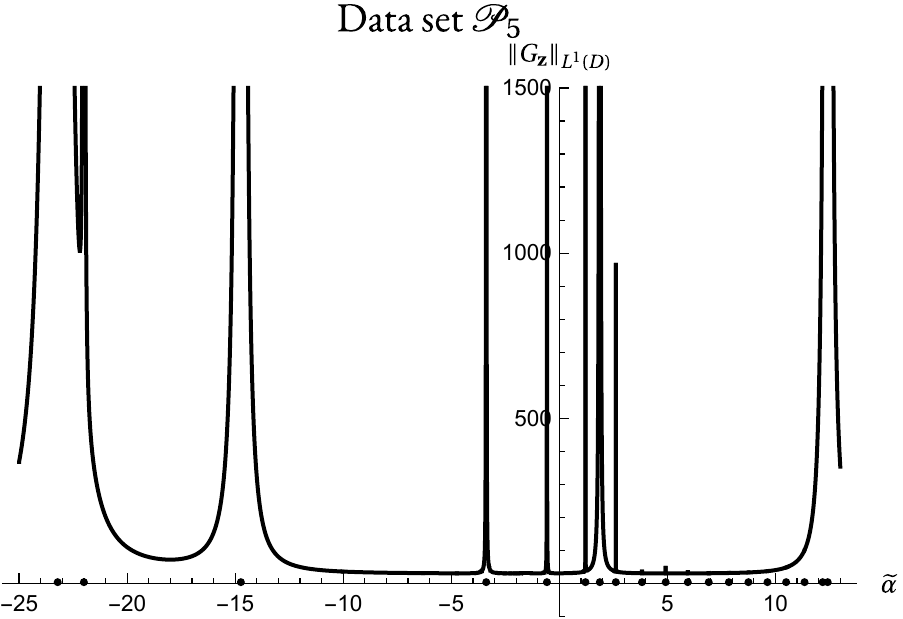}\hfill
\includegraphics[width=0.45\textwidth]{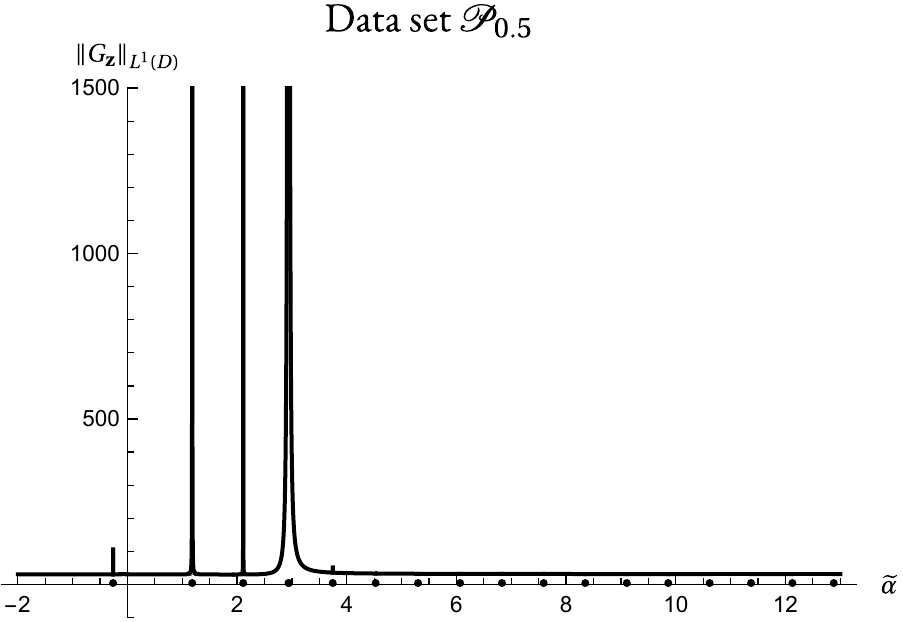}\\
\ \\\ \\
\includegraphics[width=0.45\textwidth]{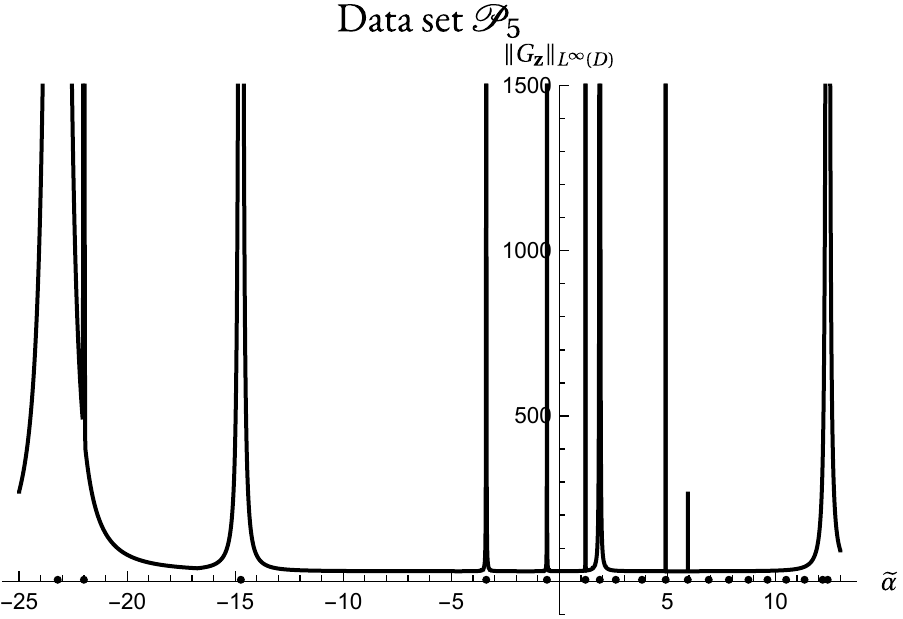}\hfill
\includegraphics[width=0.45\textwidth]{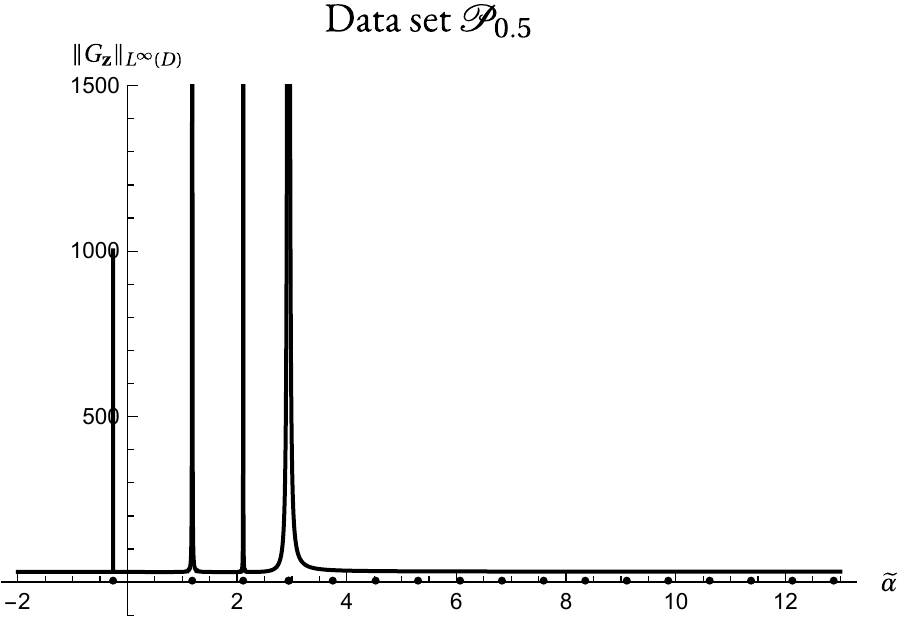}
\caption{Plots of the norms $\|G_{\zbold}\|_{L^\infty(D)}$ and  $\|G_{\zbold}\|_{L^1(D)}$ as functions of $\tilde{\alpha}$. In this and subsequent Figures, the black dots on the $\tilde{\alpha}$ axis indicate the positions of the eigenvalues of $\DtN_\Lambda$ from the corresponding column of Table \ref{tab:tab1}. \label{fig:Gnorms}} 
\end{figure}

We remark that in both cases and for both data sets the results are far from satisfactory, although better for $k=5$, and within each data set better for $L^\infty$ norm than for $L^1$ norm --- but in all the cases we do not observe peaks at quite a significant number of the nnDtN eigenvalues.

\subsection{Modified-modified far field equations} 
Instead of working with the particular modified far field equation \eqref{eq:ffeq} (or its rescaled  version involving 
the operator $\tilde{\mathcal{F}}_{\tilde{\alpha}}$) we propose working with a  \emph{modified-modified far field equation}: 
\begin{equation}\label{eq:mmffe}
\tilde{\mathcal{F}}_{\tilde{\alpha}} g_{\zbold}=\tilde{F}_{\tilde{\alpha}}\Phi^{\infty}_{\zbold} \quad \text{for } \zbold \in \overline{D},
\end{equation} 
which differs from \eqref{eq:ffeq} by an additional application of $\tilde{F}_{\tilde{\alpha}}$ in the right-hand side.

\begin{remark} Instead of using \eqref{eq:mmffe}, we may consider
\[
\tilde{\mathcal{F}}_{\tilde{\alpha}}  g_{\zbold}=F\Phi^{\infty}_{\zbold} \quad \text{for } \zbold \in \overline{D},
\]
with $F$ applied to the right-hand side rather than $\tilde{F}_{\tilde{\alpha}}$. The results in this case are very similar, and we omit them.
\end{remark}

We can now generalise \eqref{eq:mmffe} further by noting that the function $\Phi^{\infty}_{\zbold}$ in (\ref{eq:mmffe}) needs no
longer to be chosen as a far field pattern (since we create a far field pattern by applying $F_{\tilde{\alpha}}$). Taking an arbitrary $L^2(S)$ right-hand side in \eqref{eq:mmffe} we can replace the calculation of a norm of particular $g_{\zbold}$ by a calculation of the operator norm $\|\tilde{\mathcal{F}}_{\tilde{\alpha}}^{-1} \tilde{F}_{\tilde{\alpha}}\|$. 
We show the plots of these operator norms in Figure \ref{fig:norms}. The appropriate peaks at the eigenvalues of $\DtN_\Lambda$ are seen much better, and demonstrate  a substantial improvement on Figure \ref{fig:Gnorms}. 

\begin{figure}[htb!]
\centering
\includegraphics[width=0.45\textwidth]{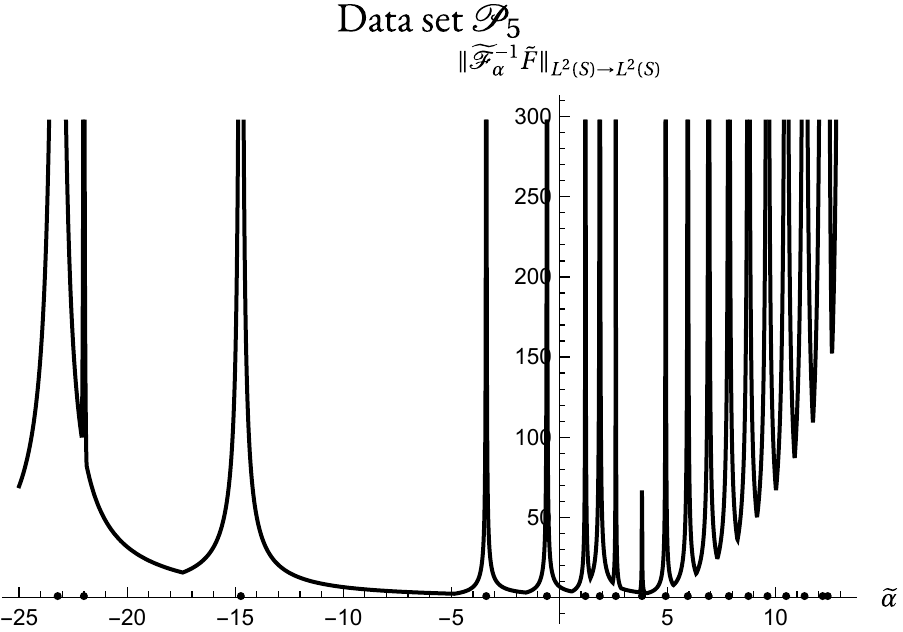}\hfill
\includegraphics[width=0.45\textwidth]{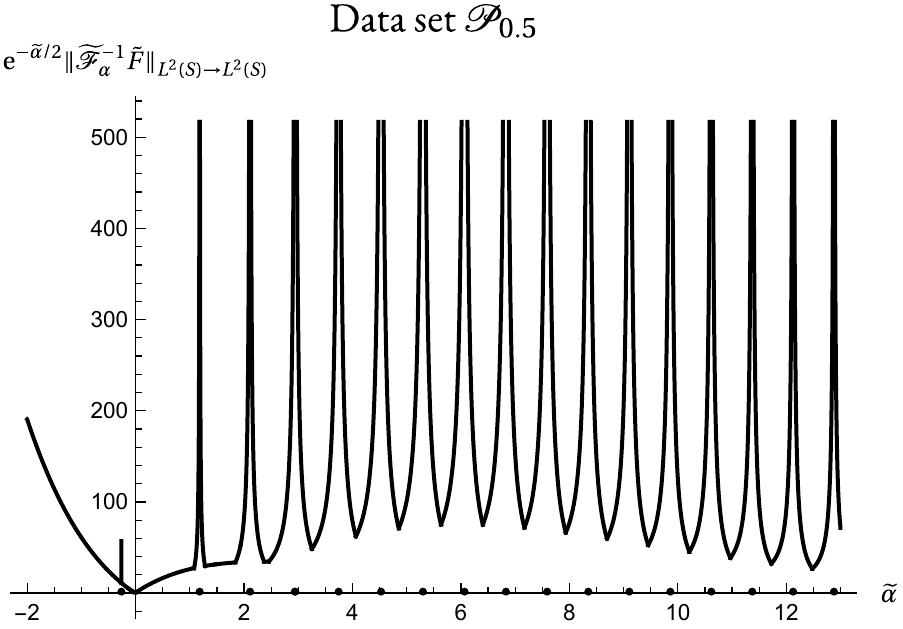}
\caption{Plots of the operator norms $\|\tilde{\mathcal{F}}_{\tilde{\alpha}}^{-1} \tilde{F}_{\tilde{\alpha}}\|_{L^2(s)\to L^2(S)}$ as functions of $\tilde{\alpha}$. Additional exponential scaling for data set $\mathcal{P}_{0.5}$ improves the appearance. \label{fig:norms}} 
 \end{figure}
 
 For further suggested modifications see \S\ref{sec:furthermod}.
 
%

\begin{appendices}
\setcounter{figure}{0}
\renewcommand{\thesection}{SM\arabic{section}}
\renewcommand{\thefigure}{SM\arabic{figure}}
\renewcommand{\thetable}{SM\arabic{table}}

\clearpage
\renewcommand{\thepage}{SM\arabic{page}}
\setcounter{page}{1}

\centerline{\bf\Large Supplementary materials}
\addcontentsline{toc}{section}{Supplementary materials}

\section{Eigenvalues of the nnDtN map for elasticity in a disk with constant density}\label{app:A}

Let $\Omega=D=\{(r\cos\theta, r\sin\theta); 0\le r<1\}$ be the unit disk in $\mathbb{R}^2$. Consider problem \eqref{eq:elasticvibrations}, \eqref{eq:elasticoperator} in $D$ and assume that $\rho_s$ is constant. In this case, we can write solutions of  \eqref{eq:elasticvibrations} using the separation of variable, see e.g. \cite{MorseFesh}, in the following manner.

Assuming for the moment $\Lambda=\omega^2>0$ (the case  $\Lambda<0$ is treated similarly, whereas the case $\Lambda=0$ requires a slightly different approach using the so-called Michell solutions of static elasticity problem, which we skip for brevity). We look for $\ubold$ in the form 
\begin{equation}\label{eq:potsol} 
\ubold = \Grad \psi_1(r,\theta)+\Curl \left(\nbold_{\mathbb{R}^2}
\psi_2(r,\theta) \right),
\end{equation}
with some unknown scalar potentials $\psi_j$, $j=1,2$, and  with $\nbold_{\mathbb{R}^2}$ being the unit vector orthogonal to the plane of the disk. 
Substituting this 
into \eqref{eq:elasticvibrations} we deduce that the scalar potentials $\psi_j$ satisfy the Helmholtz equations
\begin{equation}\label{eq:Helm} 
-\Delta \psi_j - \omega_j^2 \psi_j=0 \qquad\text{in }D ,
\end{equation}
where
\[
\omega_1:=\omega\sqrt{\frac{\rho_s}{\lambda+2\mu}},\qquad\omega_2:=\omega\sqrt{\frac{\rho_s}{\mu}}.
\]
Partial solutions of \eqref{eq:Helm} in the disk, regular at the origin,  with an angular momentum $q\in\mathbb{Z}$ are given in the standard way in terms of cylindrical Bessel functions by
\begin{equation}\label{eq:bes} 
\psi_{j,q}= 
J_{|q|}(\omega_j r)\exp(\imath q\theta) .
\end{equation}

For each of the problems considered in Section \ref{sec:MichaelOnEigs}, we now obtain the equations for eigenvalues corresponding to each angular momentum $q$ by substituting \eqref{eq:potsol}--\eqref{eq:bes}  into the appropriate boundary conditions. In particular, for eigenvalues of the nnDtN map, after some straightforward calculations we obtain that each $q\in\mathbb{Z}$ contributes to the  $\Spec(\DtN_{\Lambda})$, which is given by
\[
\alpha_q=\frac{A_q}{B_q},
\]
with
\begin{align*}
A_q&=J_{|q|}(\omega_1)J_{|q|}(\omega_2)\left(-\omega _1^2 (\lambda +2 \mu ) \left(2 (| q| -1) | q| -\omega _2^2\right)-2 \mu  \omega _2^2 (| q| -1) | q|\right)\\
&+J_{|q|}(\omega_1)J_{|q|+1}(\omega_2) 2 \omega _2 \left(2 \mu  | q| ^3-2 \mu  | q| -\omega _1^2 (\lambda +2 \mu )\right)\\
&+J_{|q|+1}(\omega_1)J_{|q|}(\omega_2) 2 \mu  \omega _1 \left(2 | q| ^3-2 | q| -\omega _2^2\right)\\
&-J_{|q|+1}(\omega_1)J_{|q|+1}(\omega_2) 4 \mu  \omega _1 \omega _2 \left(| q| ^2-1\right)
\end{align*}  
and 
\begin{align*}
B_q&=-J_{|q|}(\omega_1)J_{|q|}(\omega_2)| q| \omega _ 2^2+J_{|q|}(\omega_1)J_{|q|+1}(\omega_2) 2 \omega _2 | q|\\ 
&+J_{|q|+1}(\omega_1)J_{|q|}(\omega_2)\omega _1 \left(2 | q| +\omega _2^2\right)-J_{|q|+1}(\omega_1)J_{|q|+1} 2 \omega _1 \omega _2.
\end{align*}  
We note that the contributions of $\pm q$ coincide, therefore all eigenvalues associated to $q\ne 0$ will have multiplicity at least two.

For $\Lambda=0$, the expressions for the eigenvalues of the nnDtN map may be obtained either by taking the limit in the formulae above, or by a separate analysis of stationary solutions. Either approach leads to  
\[
\alpha_q=\frac{2(|q|^2-1)\mu(\lambda+\mu)}{|q|\lambda+(2|q|-1)\mu}.
\]
Notice that, here again, contributions of $\pm q$ coincide.

Table \ref{tab:tab1} lists, for both sets of parameters from \S\ref{sec:NumSteklovFS}, the eigenvalues $\alpha_q$ of the nnDtN map $\DtN_\Lambda$ whose value is less than $13$.

\setlength{\tabcolsep}{15pt}
\begin{table}[htb!]
\centering
\begin{tabular}{c c | c c}
\multicolumn{2}{c|}{$\mathcal{P}_{5}$}&\multicolumn{2}{c}{$\mathcal{P}_{0.5}$}\\
\multicolumn{1}{c}{$\alpha_q$}&\multicolumn{1}{c|}{$|q|$}&\multicolumn{1}{c}{$\alpha_q$}&\multicolumn{1}{c}{$|q|$}\\\hline
 -23.2133 & 0 & -0.2610 & 1 \\
 -22.0001 & 4 & 1.1820 & 2 \\
 -14.7364 & 3 & 2.1116 & 3 \\
 -3.3830 & 5 & 2.9372 & 0 \\
 -0.5739 & 6 & 2.9456 & 4 \\
 1.2091 & 7 & 3.7447 & 5 \\
 1.8721 & 1 & 4.5271 & 6 \\
 2.6131 & 8 & 5.3000 & 7 \\
 3.8243 & 9 & 6.0672 & 8 \\
 4.9210 & 10 & 6.8305 & 9 \\
 5.9430 & 11 & 7.5911 & 10 \\
 6.9132 & 12 & 8.3499 & 11 \\
 7.8456 & 13 & 9.1071 & 12 \\
 8.7495 & 14 & 9.8633 & 13 \\
 9.6313 & 15 & 10.6185 & 14 \\
 10.4957 & 16 & 11.3731 & 15 \\
 11.3460 & 17 & 12.1271 & 16 \\
 12.1846 & 18 & 12.8806 & 17 \\
 12.4194 & 2 &  &  \\
\end{tabular}
\caption{The first few eigenvalues of the nnDtN map $\DtN_\Lambda$ for the two sets of parameters. {The eigenvalues with $|q|>0$ are double, and those with $q=0$ are single.}\label{tab:tab1}}
\end{table}

To obtain the formulae for the impedance eigenvalues (that is, those of the nnNtD map up to a scaling factor), we take reciprocals of the expressions for the eigenvalues of the nnDtN map.

Finding the eigenvalues for the Neumann or mixed problems is reduced to solving, for each $q$, some transcendental equations in $\Lambda$. We omit the details.  

\section{The modified far field operator for a disk with constant density}\label{app:B}

The procedure for obtaining explicit expression for the modified far filed operator in case of the unit disk is fairly standard, and we mostly follow Colton and Kress here.

We start by looking at the problem \eqref{eq:PM1}. Let $\dbold=(\cos\phi, \sin\phi)$ and work in polar coordinates $(r, \theta)$ in the unit disk $D$.  We have, through the usual expansion of a plane wave in spherical waves,
\[
p^i((r, \theta); \dbold)=\mathrm{e}^{\imath k \dbold\cdot \xbold}=\mathrm{e}^{\imath k r\cos(\theta-\phi)}=\sum_{q=-\infty}^\infty \imath^q J_q(kr) \mathrm{e}^{\imath q(\theta-\phi)}.
\]  
We also have for the solution of Helmholtz equation satisfying Sommerfield's radiation condition,
\[
p^s((r, \theta); \dbold)=\sum_{q=-\infty}^\infty \hat{p}_q \imath^q H_q^{(1)}(kr) \mathrm{e}^{\imath q(\theta-\phi)},
\]
with coefficients $\hat{p}_q$ still to be determined. The latter yields, via standard asymtotics of Hankel functions, 
\[
p^\infty(\theta; \dbold)=\mathrm{e}^{-\imath\pi/4}\sqrt{\frac{2}{\pi k}} \sum_{q=-\infty}^\infty \hat{p}_q \mathrm{e}^{\imath q(\theta-\phi)}.
\]

Acting in a similar manner for the problem \eqref{PM_Steklov}, we obtain 
\[
h^s((r, \theta); \dbold)=\sum_{q=-\infty}^\infty \hat{h}_q \imath^n H_q^{(1)}(kr) \mathrm{e}^{\imath q(\theta-\phi)},
\]
with coefficients $\hat{h}_q=\hat{h}_{q,\alpha}$ still to be determined and which give
\[
h^\infty(\theta; \dbold)=\mathrm{e}^{-\imath\pi/4}\sqrt{\frac{2}{\pi k}} \sum_{q=-\infty}^\infty \hat{h}_q \mathrm{e}^{\imath q(\theta-\phi)}.
\]
Finding the coefficients $\hat{h}_q$ is easier by substituting into the boundary condition \eqref{PM_Steklov_b} and separating the harmonics:
\[
\hat{h}_{q}=-\imath^q\frac{kJ_q'(k)+\alpha J_q(k)}{k {H_q^{(1)}}'(k)+\alpha H_q^{(1)}(k)}.
\]

To evaluate coefficients $\hat{p}_q$, we substitute the expansion of $p^s$ into \eqref{PM1b} and use the fact that $\mathrm{e}^{\imath q \theta}$ are the eigenfunctions of the nnNtD map corresponding to the eigenvalues $\frac{1}{\alpha_{|q|}}$, 
\[
\NtD_{\Lambda}\left(\mathrm{e}^{\imath q \theta}\right)=\frac{1}{\alpha_{|q|}}\mathrm{e}^{\imath q \theta},
\]
that yields
\[
\hat{p}_q=-\imath^q\frac{kJ_q'(k)+\frac{\rho_f c_f^2 k^2}{\alpha_{|q|}} J_q(k)}{k {H_q^{(1)}}'(k)+\frac{\rho_f c_f^2 k^2}{\alpha_{|q|}} H_q^{(1)}(k)}.
\]

Combining everything together for
\[
g(\theta)= \sum_{q=-\infty}^{\infty} \hat{g}_q \mathrm{e}^{\imath q \theta}\in L^2([0,2\pi])
\]
(or equivalently for $\{\hat{g}_q\}\in \ell^2(\mathbb{Z})$), we obtain
\begin{align*}  
(F g)(\theta)&=\int_0^{2\pi} p^\infty(\theta; \phi)\,g(\phi)\,\mathrm{d}\phi=\mathrm{e}^{-\imath\pi/4}\sqrt{\frac{8\pi}{k}}\sum_{q=-\infty}^\infty\hat{p}_q \mathrm{e}^{\imath q \theta},\\
(F_\alpha g)(\theta)&=\int_0^{2\pi} h^\infty(\theta; \phi)\,g(\phi)\,\mathrm{d}\phi=\mathrm{e}^{-\imath\pi/4}\sqrt{\frac{8\pi}{k}}\sum_{q=-\infty}^\infty \hat{h}_q \mathrm{e}^{\imath m \theta},\\
(\mathcal{F}_\alpha g)(\theta)&=\int_0^{2\pi} (p^\infty(\theta; \phi)-h^\infty(\theta; \phi))\,g(\phi)\,\mathrm{d}\phi=\mathrm{e}^{-\imath\pi/4}\sqrt{\frac{8\pi}{k}}\sum_{q=-\infty}^\infty \hat{f}_{\alpha, q} \hat{g}_q \mathrm{e}^{\imath q \theta},
\end{align*} 
where 
\[
\hat{f}_{\alpha, q} := \hat{p}_q-\hat{h}_q.
\]

It is convenient to switch to the parameter $\tilde{\alpha}$ as in \eqref{eq:tildealpha} and set, in addition to \eqref{eq:tildeFdefn},
\begin{equation}\label{eq:tildefalpha}
\begin{aligned}
F_{\tilde{\alpha}}:=F_{\rho_f \omega^2/\tilde{\alpha}}:\, g(\theta)&\mapsto \mathrm{e}^{-\imath\pi/4}\sqrt{\frac{8\pi}{k}}\sum_{q=-\infty}^\infty \tilde{h}_{\tilde{\alpha},q} \hat{g}_q \mathrm{e}^{\imath q \theta},\\
\tilde{\mathcal{F}}_{\tilde{\alpha}}:\, g(\theta)&\mapsto \mathrm{e}^{-\imath\pi/4}\sqrt{\frac{8\pi}{k}}\sum_{q=-\infty}^\infty \tilde{f}_{\tilde{\alpha},q} \hat{g}_q \mathrm{e}^{\imath q \theta},
\end{aligned}
\end{equation}
where we have denoted 
\begin{equation}\label{eq:tildefph}
\begin{aligned}
\tilde{f}_{\tilde{\alpha},q}:=\hat{f}_{\rho_f c_f^2 k^2/\tilde{\alpha}, q}&=\tilde{p}_q-\tilde{h}_q,\\
\tilde{p}_q=\hat{p}_q&=-\imath^q\frac{\alpha_{|q|}J_q'(k)+\rho_f c_f^2 k J_q(k)}{\alpha_{|q|} {H_q^{(1)}}'(k)+\rho_f c_f^2 k H_q^{(1)}(k)},\\
\tilde{h}_q=\tilde{h}_{q,\tilde{\alpha}}:=\hat{h}_{\rho_f c_f^2 k^2/\tilde{\alpha}, q}&=-\imath^q\frac{\tilde{\alpha}J_q'(k)+ \rho_f c_f^2  k J_q(k)}{\tilde{\alpha} {H_q^{(1)}}'(k)+\rho_f c_f^2 k H_q^{(1)}(k)}.
\end{aligned}
\end{equation}
The various operator norms plotted in \S\ref{sec:NumSteklovFS} are then easily evaluated.

Generally speaking, we are looking at solutions of equations of the type
\begin{equation}\label{eq:tildefeqn}
\tilde{\mathcal{F}}_{\tilde{\alpha}}g(\theta)=\Psi(\theta),\quad\theta\in[0,2\pi),
\end{equation}
with different right-hand sides 
\[
\Psi(\theta)=\sum_{q=-\infty}^\infty \hat{\psi}_q  \mathrm{e}^{\imath q \theta}\in L^2([0,2\pi]).
\]
Taking into account \eqref{eq:tildefalpha} and \eqref{eq:tildefph}, the solution to \eqref{eq:tildefeqn} is formally given by 
\begin{equation}\label{eq:gsolution}
g(\theta)=\tilde{\mathcal{F}}_{\tilde{\alpha}}^{-1}\Psi(\theta)= \sum_{q=-\infty}^\infty (\tilde{f}_{\tilde{\alpha},q})^{-1} \hat{\psi}_q  \mathrm{e}^{\imath q \theta}
=\mathrm{e}^{\imath\pi/4}\sqrt{\frac{k}{8\pi}}\sum_{q=-\infty}^\infty \frac{1}{\tilde{p}_q-\tilde{h}_q} \hat{\psi}_q  \mathrm{e}^{\imath q \theta}.
\end{equation}

Returning to the equation  \eqref{eq:mmffe}, where we take $\zbold=-\zeta\in[-1,0]$ as discussed,
we rewrite the right-hand side using 
\[
\Phi^{\infty}_{\zbold}(\theta)= \frac{\mathrm{e}^{\mathrm{i}\pi/4} }{\sqrt{8\pi k } }  \mathrm{e}^{\mathrm{i}k \zeta\cos\theta}
= \frac{\mathrm{e}^{\mathrm{i}\pi/4} }{\sqrt{8\pi k } }  \sum_{q=-\infty}^\infty  \mathrm{i}^q J_q(k\zeta)\mathrm{e}^{\imath q \theta},
\]
which yields, with account of \eqref{eq:tildefalpha}, \eqref{eq:tildefph} and \eqref{eq:gsolution}, the solution
\[
g_{\zbold}(\theta)=\frac{\mathrm{e}^{\mathrm{i}\pi/4} }{\sqrt{8\pi k } } \sum_{q=-\infty}^\infty  \mathrm{i}^{q} \frac{\tilde{h}_q}{\tilde{p}_q-\tilde{h}_q} J_q(k\zeta)\mathrm{e}^{\imath q \theta},
\]
with the $L^2(S)$-norm squared
\begin{equation}\label{eq:gznorm1}
G(\zbold):=\|g_{\zbold}\|^2 = \frac{1}{4k} \sum_{q=-\infty}^\infty \left|\frac{\tilde{h}_q}{\tilde{p}_q-\tilde{h}_q}\right|^2 J_q^2(k\zeta).
\end{equation} 
The $L^2$ norm of $G(\zbold)$ is evaluated using the standard relation 
\begin{equation}\label{eq:besselint}
\int_0^1 \zeta J_q^2(k\zeta)\,\mathrm{d}\zeta=\frac{1}{2}\left(J_{|q|}^2(k)-\frac{|q|}{k}J_{|q|}(k)J_{|q|+1}(k)+J_{|q|+1}^2(k)\right).
\end{equation}

\section{Further modifications}\label{sec:furthermod}

Calculating operator norms  in arbitrary geometries would be costly; to illustrate the possibility of improving the results and at the same keeping the computational costs to a minimum, we suggest the following ad-hoc approach.
 Fix a value $A$ of a variable $\tilde{\alpha}$ (some experiments may be required in order to avoid $A$ coinciding with one of the nnDtN eigenvalues), and a parameter $t>\frac{1}{2}$, and consider the equation
\begin{equation}\label{eq:FA}
\tilde{\mathcal{F}}_{\tilde{\alpha}} f_{A,t} =  \tilde{\mathcal{F}}_{A} \Psi_t,
\end{equation}
where we set 
\[
\Psi_t(\theta):=\sum_{q=-\infty}^{+\infty} \frac{1}{(|q|+1)^t}\mathrm{e}^{\imath q \theta} .
\]
Here the parameter $t$ determines how quickly the Fourier coefficients of $\Psi_t$ decrease. In principle, one can use an additional random factor $\xi_q$ (e.g.  with $|\xi_q|\in[0.5,5]$) in the definition of $\Psi_t$, but this produces very similar results and is not necessary. 

The solution of \eqref{eq:FA} is given by 
\[
 f_{A,t}=\tilde{\mathcal{F}}_{\tilde{\alpha}}^{-1} \tilde{\mathcal{F}}_{A} \Psi_t,
 \]
and we plot its norm as a function of $\tilde{\alpha}$ in Figures  \ref{fig:fnorms5} and \ref{fig:fnorms05}, taking either $A=0$ or $A=10$ and either $t=0.55$ or $t=1$. This demonstrates  a very good prediction of nnDtN eigenvalues; however, this approach  still  needs a full theoretical analysis and justification for non-circular obstacles.

\begin{figure}[htb!]
\centering
\includegraphics[width=0.45\textwidth]{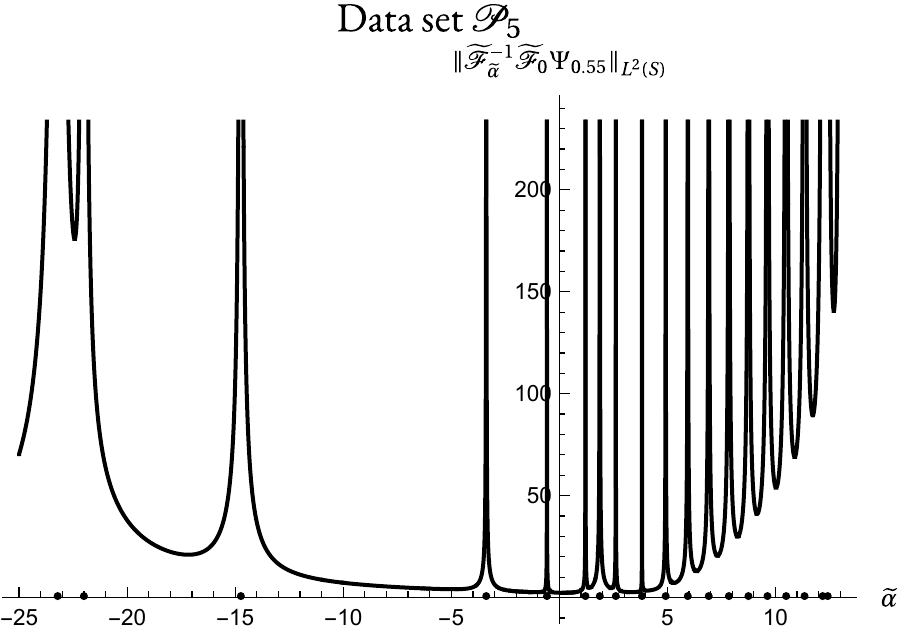}\hfill
\includegraphics[width=0.45\textwidth]{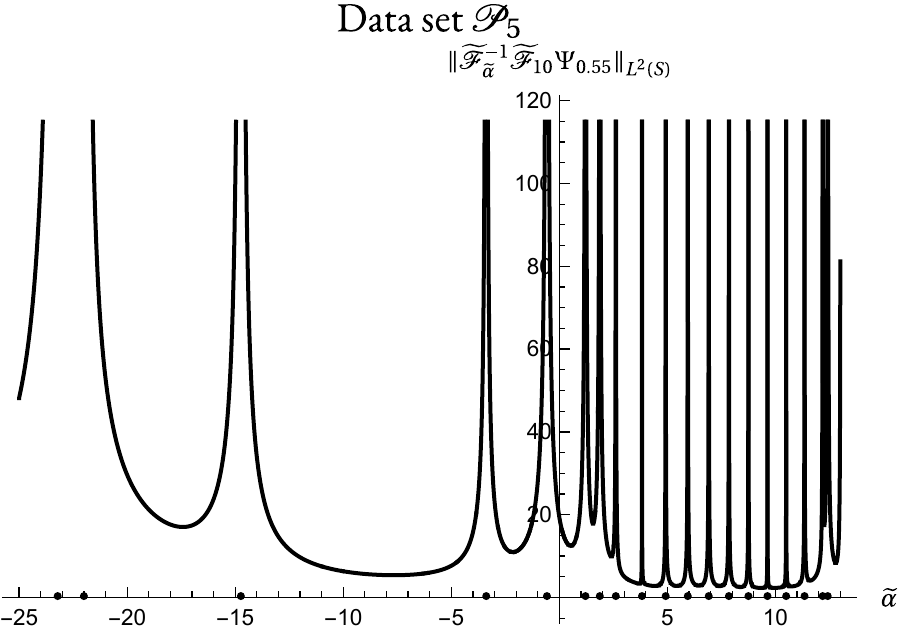}\\
\ \\\ \\ 
\includegraphics[width=0.45\textwidth]{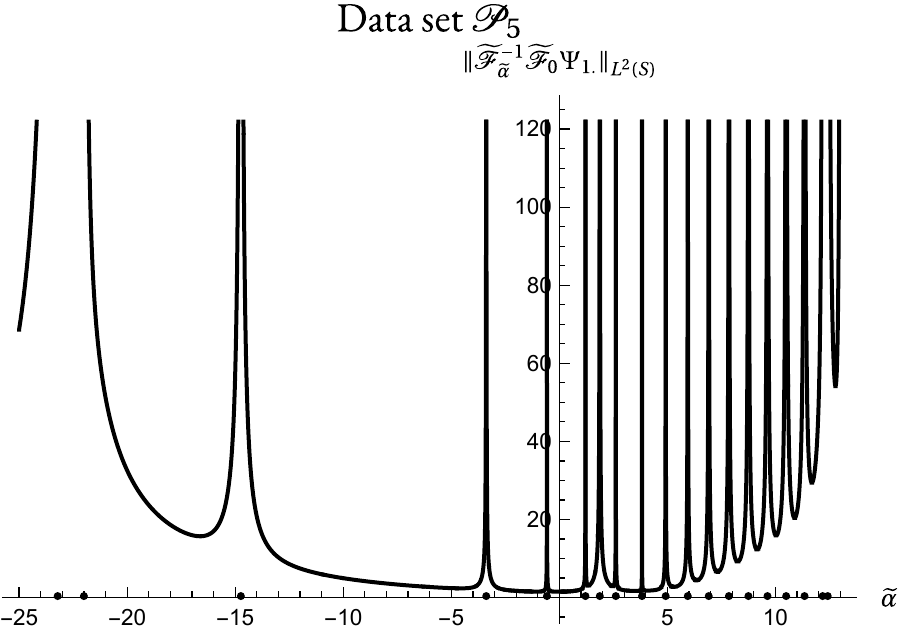}\hfill
\includegraphics[width=0.45\textwidth]{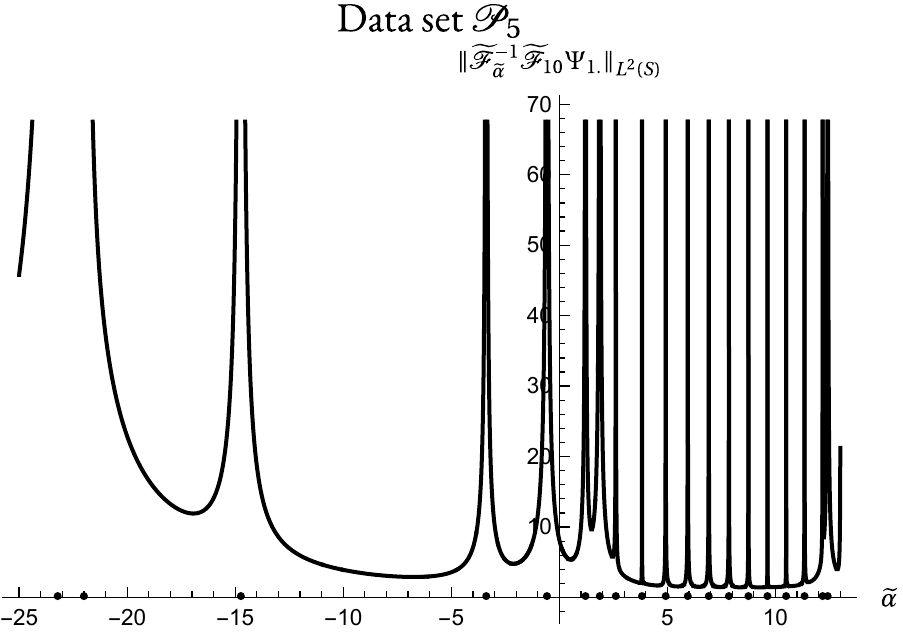}
\caption{Plots of the norms of  $f_{A,t}$ for different choices of $A$ and $t$ for dataset $\mathcal{P}_5$.\label{fig:fnorms5}} 
 \end{figure}

\begin{figure}[hbt!]
\centering
\includegraphics[width=0.45\textwidth]{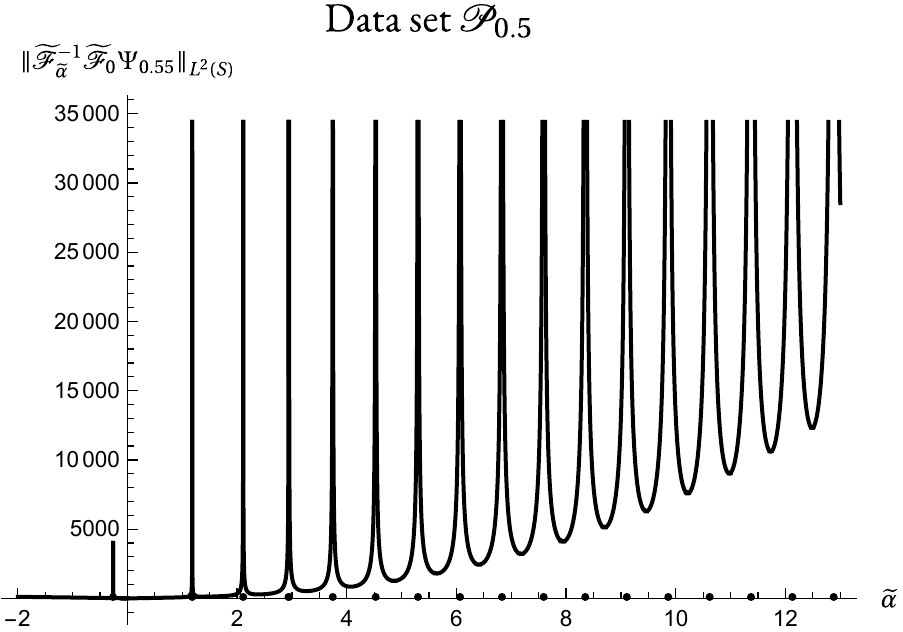}\hfill
\includegraphics[width=0.45\textwidth]{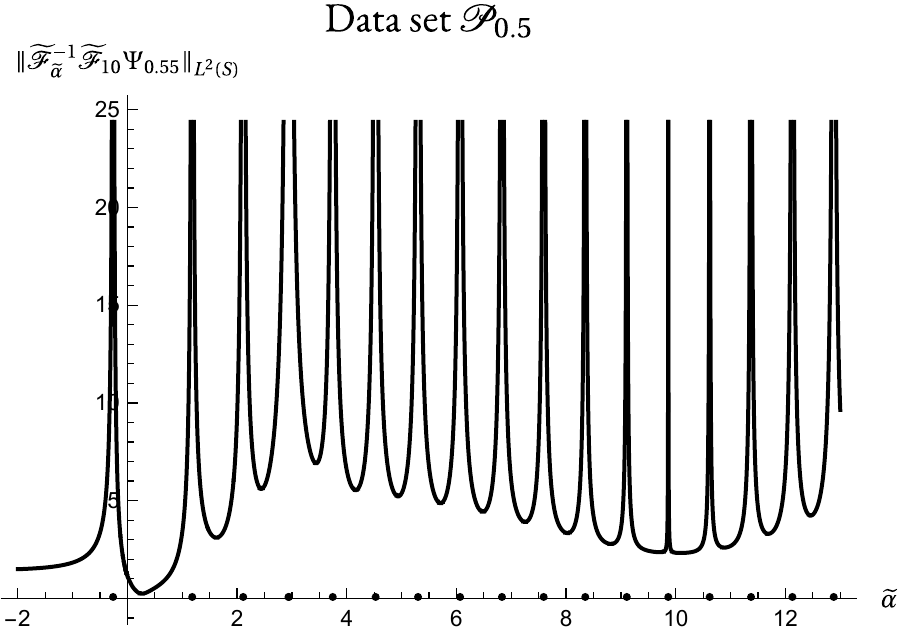}\\
\ \\\ \\
\includegraphics[width=0.45\textwidth]{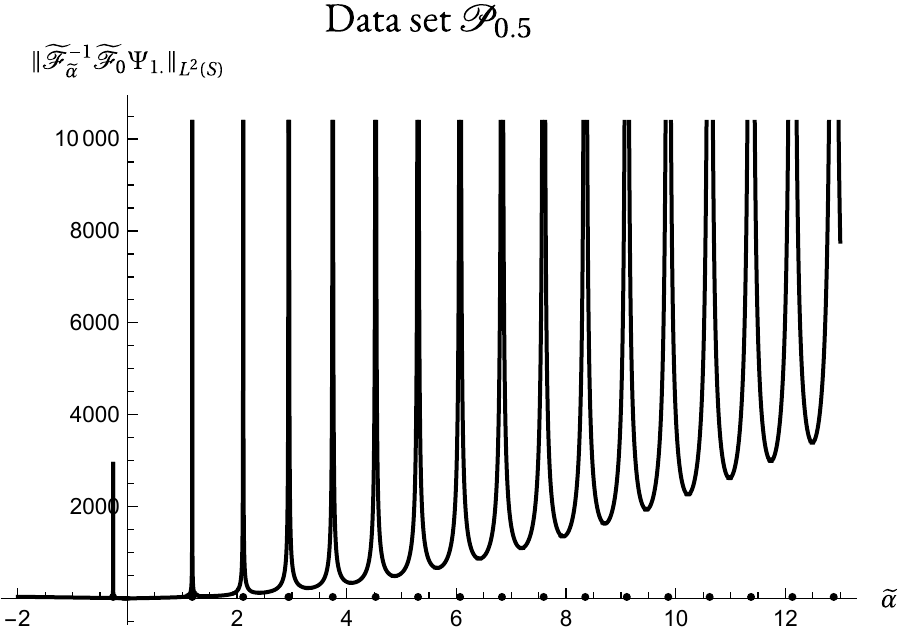}\hfill
\includegraphics[width=0.45\textwidth]{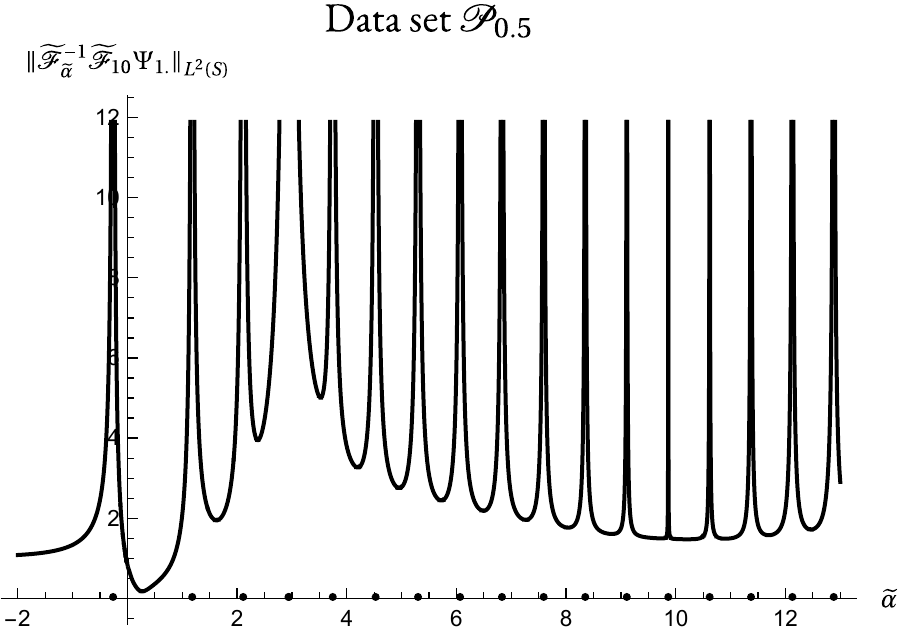}
\caption{Plots of the norms of  $f_{A,t}$ for different choices of $A$ and $t$ for dataset $\mathcal{P}_{0.5}$.\label{fig:fnorms05}} 
 \end{figure}

\clearpage
\end{appendices}
\end{document}